\documentclass[11pt,reqno]{amsart}

\usepackage{mathrsfs}
\usepackage{amssymb}
\usepackage{amsthm}
\usepackage{amsmath,amsfonts,amssymb,esint}
\usepackage{graphics,color}
\usepackage{enumerate}

\usepackage{marginnote}
\usepackage{xfrac}
\usepackage{mathtools}      
\usepackage{hyperref}

\usepackage{comment}

\newtheorem{theorem}{Theorem}[section]
\newtheorem{lemma}[theorem]{Lemma}
\newtheorem{proposition}[theorem]{Proposition}
\newtheorem{corollary}[theorem]{Corollary}
\newtheorem{definition}[theorem]{Definition\rm}
\newtheorem{remark}{Remark}

\newcommand*{\R}{\ensuremath{\mathbb{R}}}

\newcommand*{\N}{\ensuremath{\mathbb{N}}}

\newcommand*{\C}{\ensuremath{\mathbb{C}}}

\newcommand{\bA}{\mathbf{A}}
\newcommand{\bB}{\mathbf{B}}
\newcommand{\bC}{\mathbf{C}}
\newcommand{\sym}{{\mathrm{Sym}}\,}
\newcommand{\sC}{\mathscr{C}}
\newcommand{\sS}{\mathscr{S}}
\newcommand{\sD}{\mathscr{D}}
\newcommand{\sL}{\mathscr{L}}
\newcommand{\zb}{\overline{z}}

\newcommand{\vertiii}[1]{{\left\vert\kern-0.25ex\left\vert\kern-0.25ex\left\vert #1 
    \right\vert\kern-0.25ex\right\vert\kern-0.25ex\right\vert}}

\begin{document}

\title[Nash-Kuiper theorem in $C^{\sfrac{1}{5}-\delta}$]
{A Nash-Kuiper theorem for $C^{1,\sfrac{1}{5}-\delta}$ immersions of surfaces in $3$ dimensions}

\author[De Lellis]{Camillo De Lellis}
\address{Institut f\"ur Mathematik, Universit\"at Z\"urich, CH-8057 Z\"urich}
\email{camillo.delellis@math.uzh.ch}

\author[Inauen]{Dominik Inauen}
\address{Institut f\"ur Mathematik, Universit\"at Z\"urich, CH-8057 Z\"urich}
\email{dominik.inauen@uzh.ch}

\author[Sz\'ekelyhidi]{L\'aszl\'o Sz\'ekelyhidi Jr.}
\address{Institut f\"ur Mathematik, Universit\"at Leipzig, D-04109 Leipzig}
\email{laszlo.szekelyhidi@math.uni-leipzig.de}

\maketitle

\begin{abstract} 
We prove that, given a $C^2$ Riemannian metric $g$ on the $2$-dimensional disk $D_2$, any short $C^1$ immersion of $(D_2,g)$ into $\R^3$ can be uniformly approximated with $C^{1,\alpha}$ isometric immersions for any $\alpha < \frac{1}{5}$. This statement improves previous results by Yu.F.~Borisov and of a joint paper of the first and third author with S.~Conti.
\end{abstract}

\section{Introduction}

In this paper we consider isometric immersions of $2$-dimensional disks in $\R^3$.  With $D_r (x_0)$ and $\overline{D}_r (x_0)$ we denote, respectively, the open and closed disks in $\R^2$ with center $x_0$ and radius $r$. When $x_0=0$ we write simply $D_r$, resp. $\overline{D}_r$. If $g$ is a $C^0$ Riemannian metric on $\overline{D}_r (x_0)$, an isometric immersion $u: \overline{D}_r (x_0)\to \R^n$ is a $C^1$ immersion  such that $u^\sharp e = g$, where $e$ denotes the Euclidean metric on $\R^n$. In other words this means that
\begin{equation}
\partial_i u \cdot \partial_j u = g_{ij}\, .
\end{equation}
If $0 < \partial_i u \cdot \partial_j u < g_{ij}$ in the sense of quadratic forms, we then call $u$ a short immersion. Our main theorem is the following approximation result which, using a popular terminology, is an ``$h$-principle'' statement, cf. \cite{Gromov86,SzLecturenotes,DSsurvey}.

\begin{theorem}\label{t:main}
Let $g$ be a $C^2$ metric on $\overline{D}_2$ and $\bar u\in C^1 (\overline{D}_2, \R^3)$ a short immersion. For every $\delta>0$ and $\varepsilon > 0$ there is a $C^{1, \sfrac{1}{5}-\delta}$ isometric immersion $u$ of $(\overline{D}_1, g)$ in $\R^3$ such that $\|\bar u-u\|_{C^0} < \varepsilon$. 
\end{theorem}

The well-known ground-breaking result of Nash and Kuiper \cite{Nash54,Kuiper55} implies that Theorem \ref{t:main} holds with $C^1$ replacing $C^{1, 1/5-\delta}$. The first extension to the $C^{1,\alpha}$ category was {obtained} by Yu.F.~Borisov: {in \cite{Borisov65} a version of Theorem \ref{t:main}  for $C^{1,1/7-\delta}$ immersions (and embeddings) of $2$-dimensional disks with real analytic metrics $g$ was announced; in fact, more generally, the theorem in \cite{Borisov65} applied to $C^{1, \alpha}$ isometric embeddings of $n$-dimensional balls in $\R^{n+1}$ under the assumption that $\alpha < \frac{1}{1+n+n^2}$.  In \cite{Borisov2004} a detailed proof for the case of $2$-dimensional disks and with exponents $\alpha < \frac{1}{13}$ appeared}. In the paper \cite{CDS} the first and third author jointly with S.~Conti gave {a detailed and self-contained proof of all the statements contained in \cite{Borisov65}} for $C^2$ metrics on $n$-dimensional balls (with the same H\"older exponents) and analogous  generalizations to $C^{1,\beta}$ metrics on compact manifolds without topological restrictions. For the optimal H\"older exponents in the case of rough metrics, which depend on $\beta$, the dimension and the topology of the manifold, we refer to \cite{CDS}. 

The main contribution of this paper {is to be able to raise the optimal H\"older exponent} for $2$-dimensional disks from $\frac{1}{7}$ to $\frac{1}{5}$, by taking advantage of the theory of conformal maps. {The question of the optimal exponent for which an h-principle statement as in Theorem \ref{t:main} can hold is also relevant for rigidity theory, as we will explain below.}

\medskip

It is known that Theorem \ref{t:main} cannot hold for $C^{1,\alpha}$ immersions $u$ when $\alpha>\frac{2}{3}$ and $g$ has positive Gauss curvature: under these assumptions it was shown by Borisov that $u(\overline{D}_1)$ must be (a portion of) a convex surface. This was the outcome of a series of papers, cf. \cite{BorisovRigidity1,BorisovRigidity2},
and an alternative shorter proof has been given in \cite{CDS}. Borisov's theorem extends the classical rigidity result for the Weyl problem: if $(S^2,g)$ is a 
compact Riemannian surface with positive Gauss curvature and $u\in C^2$ is an isometric immersion into $\R^3$, then $u$ is uniquely determined up to a rigid motion (\cite{CohnVossen,Herglotz}, see also \cite{Spivak5} for a thorough discussion). 

The technique used to prove approximation results as in Theorem \ref{t:main} follows an iteration scheme called convex integration. The latter was developed by Gromov \cite{Gromov73,Gromov86} into a very powerful tool to prove $h$-principle statements in a wide variety of geometric problems (see also \cite{Eliashberg,Spring}). In general the regularity of solutions obtained via convex integration agrees with the highest derivatives appearing in the equations (see \cite{SpringRegularity}). Thus, an interesting question raised in \cite{Gromov86} p219 is how one could extend the methods to produce more regular solutions. 
Essentially the same question, in the case of isometric embeddings,
is also mentioned in \cite{Yau}, see Problem 27. In particular, it is tempting to imagine the existence of a threshold $\alpha_0$ so that:
\begin{itemize}
\item the $h$-principle holds for isometric $C^{1, \alpha}$ immersions of $2$-dimensional disks in $\R^3$ whenever $\alpha<\alpha_0$;
\item rigidity holds for $C^{1, \alpha}$ immersions of positively curved $2$-dimensional disks in $\R^3$ whenever $\alpha > \alpha_0$.
\end{itemize}
Hence a summary of our current knowledge is that, if such a threshold $\alpha_0$ exists, then it must lie in the interval $[\frac{1}{5}, \frac{2}{3}]$. 

\medskip

Starting with the work \cite{DS}, the first and third author pointed out a surprising similarity between the latter question and a long-standing conjecture in the theory of turbulence: in \cite{Onsager} Onsager conjectured the existence of a threshold H\"older regularity discriminating the validity of the conservation of kinetic energy for weak solutions of the incompressible Euler equations. The ``rigidity'' part of Onsager's conjecture was established by Eyink and Constantin, E and Titi in the papers \cite{Eyink} and \cite{CET}. The paper \cite{DS2} gave the first proof of the existence of continuous solutions that violate the conservation of total kinetic energy. A series of subsequent works \cite{DS3, Isett, BDS, Buckmaster, BDIS, BDS2} have made a quite substantial progress in settling Onsager's conjecture.

\medskip

Theorem \ref{t:main} could be improved in several directions. In particular, with little additional technicalities, which we believe to be of secondary importance, we will also show the following

\begin{theorem}\label{t:main2}
Let $g$ be a $C^2$ metric on $\overline{D}_1$ and $\bar u\in C^1 (\overline{D}_1, \R^3)$ a short immersion. For every $\delta>0$ and $\varepsilon > 0$ there is a $C^{1, \sfrac{1}{5}-\delta}$ isometric immersion $u$ of $(\overline{D}_1, g)$ in $\R^3$ such that $\|\bar u-u\|_{C^0} < \varepsilon$. If in addition $\bar u$ is an embedding, then $u$ can be chosen to be an embedding.
\end{theorem}

\section{Main iteration}

Theorem \ref{t:main} is achieved via an iteration, which depends upon several parameters. We start introducing the main ones.
The first parameter $\alpha>0$ is an exponent, which is assumed to be rather small, in fact smaller
than a geometric constant:
\begin{equation}\label{e:alpha}
0<\alpha < \alpha_0\, .
\end{equation}
Two further exponents will be called $c$ and $b$, both assumed to be larger than $1$, and a basis $a$, assumed to be very large. We then define the parameters
\begin{equation}\label{e:delta_lambda}
\delta_q := a^{-b^q}\qquad \lambda_q:= a^{cb^{q+1}}\, ,
\end{equation}
where $q$ is an arbitrary natural number. $b$ can in fact be chosen rather close to $1$: how much it is allowed to be close to $1$ depends on how close is $\alpha$ to $0$. $c$ will be larger but rather close to $\frac{5}{2}$, depending on how close are $b-1$ and $\alpha$ to $0$. More precisely, we summarize the conditions which $b$ and $c$ need to satisfy in the following two inequalities
\begin{align}
 \frac{3}{2} >\; & b >\frac{2}{(2-\alpha)(1-2\alpha)}\label{e:abc_1}\\
&c > \frac{2(2-\alpha) b^2 - (3-2\alpha) b -1}{b((2-\alpha) (1-2\alpha) b -2)}  = \frac{((4-2\alpha) b +1) (b-1)}{b ((2-5\alpha + 2\alpha^2) b -2)}\, .\label{e:abc_2}
\end{align}
It is moreover convenient to introduce the notation
\begin{equation}\label{e:g_q}
g_q := g - \delta_{q+1} e\, ,
\end{equation}
which simplifies several formulas.

\begin{proposition}\label{p:main}
Fix a metric $g$ as in Theorem \ref{t:main}. There is a positive constant $\alpha_0$ such that for every $\alpha$ as in \eqref{e:alpha} 
we can choose positive numbers $\sigma_0 (\alpha) < 1$ and $C_0$ with the following property. Assume $b$ and $c$ satisfy \eqref{e:abc_1} and \eqref{e:abc_2}, fix any $\bar{C}\geq C_0$ and assume that $\lambda_q$ and $\delta_q$ are defined as in \eqref{e:delta_lambda}, where $a$ is sufficiently large depending on $\alpha, b,c,g,\bar{C}$, namely
\begin{align}
& a > a_0 (\alpha, b, c, g, \bar C)\, .\label{e:a_pretty_large}
\end{align}
If $q\in \N$ and $u_q: \overline{D}_{1+2^{-q-1}} \to \R^3$ is an immersion such that
\begin{align}
&\|g_q - u^\sharp_q e\|_\alpha \leq \sigma_0 \delta_{q+1} \label{e:me_0_alpha_q}\\
&\|D^2 u_q\|_{0} \leq \bar C \delta_q^{\sfrac{1}{2}} \lambda_q\, ,\label{e:Du_2_q}
\end{align}
then there is an immersion $u_{q+1}: \overline{D}_{1+2^{-q-2}} \to \R^3$ such that
\begin{align}
&\|g_{q+1} - u^\sharp_{q+1} e\|_0 \leq \frac{\sigma_0}{3} \delta_{q+2} \lambda_{q+1}^{-\alpha}\label{e:me_0_q+1}\\
&\|D (g_{q+1}- u^\sharp_{q+1} e)\|_0 \leq \frac{\sigma_0}{3} \delta_{q+2} \lambda_{q+1}^{1-\alpha}\label{e:me_1_q+1}\\
&\|u_q-u_{q+1}\|_0 \leq \delta_{q+1}^{\sfrac{1}{2}} \lambda_{q+1}^{-\gamma}\label{e:u_0_q+1}\\
&\|D(u_q-u_{q+1})\|_0 \leq C_0 \delta_{q+1}^{\sfrac{1}{2}}\label{e:u_1_q+1}\\
&\|D^2 u_{q+1}\|_0 \leq \bar C \delta_{q+1}^{\sfrac{1}{2}} \lambda_{q+1}\label{e:u_2_q+1}\, ,
\end{align}
where $\gamma = \gamma (\alpha, b,c)>0$. 
\end{proposition}

As already mentioned, Proposition \ref{p:main} will be used in an iteration scheme to show Theorem \ref{t:main}.
The reader will notice that the starting assumption \eqref{e:me_0_alpha_q} does not exactly match the conclusions
\eqref{e:me_0_q+1}-\eqref{e:me_1_q+1}. On the other hand, a simple interpolation shows that \eqref{e:me_0_q+1} and \eqref{e:me_1_q+1} together imply the estimate 
\[
\|g_{q+1}-u_{q+1}^\sharp e\|_\alpha \leq \sigma_0 \delta_{q+2}\, ,
\] 
which corresponds to \eqref{e:me_0_alpha_q} at the next step of the iteration. It is possible to state a version of Proposition \ref{p:main} where the assumptions and conclusions look more homogeneous, but there would be no real simplification neither in the statement nor in the proof.

\medskip

Observe that, by our condition upon the parameters, $u_q$ is obviously a short map, because we have
\[
u_q^\sharp e \leq g_q  + \sigma_0 \delta_{q+1} e = g - (1-\sigma_0) \delta_{q+1} e < g\, ,
\]
where all the inequalities are understood in the sense of quadratic forms. Thus, as a simple corollary we know that
\begin{equation}\label{e:uniform_C1}
\|Du_q\|_{C^0} \leq C 
\end{equation}
for some constant $C$ which only depends upon $g$. 

\medskip

As in the Nash-Kuiper classical theorem, the map $u_{q+1}$ is obtained from the map $u_q$ by adding a certain number of perturbations, each consisting of highly oscillatory functions.
As it is clear from the arguments in \cite{CDS}, the threshold H\"older exponent that can be reached by a Nash-Kuiper type iteration is $\frac{1}{1+2 n_\star}$, where $n_\star$ is the number of such perturbations. Each perturbation adds, modulo small error terms, a smooth symmetric rank-$1$ tensor, called ``primitive metric'', to $u_q^\sharp e$.
$n_\star$ is then the smallest number of summands needed to write the metric error $g - u_q^\sharp e$ as a (positive) linear combination of such ``primitive metrics''. 

We know by the inductive assumption that $(g- u_q^\sharp e)/\|g-u_q^\sharp e\|_0$ is close to $e$, which implies that $n_\star$ can be chosen to be the dimension of the space of symmetric matrices. Thus, if $n$ is the dimension of the manifold, $n_\star = \frac{n (n+1)}{2}$: this explains the threshold $\frac{1}{1+2n_\star} = \frac{1}{1+n+n^2}$ reached in \cite{CDS} and claimed originally by Borisov. In particular in dimension $2$ the number $n_\star$ equals $3$ and Borisov's threshold is $\frac{1}{7}$. 

\medskip

The starting point of this paper is the simple observation that in $2$ dimensions we can use a conformal change of coordinates to diagonalize $g- u^\sharp e$ and hence reduce the number $n_\star$ from $3$ to $2$: this justifies the new threshold $\frac{1}{5}$. However, the regularity of the change of coordinates needed to implement this idea deteriorates with $q$ and thus it is not at all clear that the method really improves the regularity of the final map. In fact at first it is not even clear that the new iteration scheme yields any $C^{1,\alpha}$ regularity at all. 

In order to overcome this difficulty we obviously need to estimate quite carefully several norms of the conformal change of coordinates, at each step: for this reason we need to keep track of some H\"older norm of $g- u_q^\sharp e$. However, to ensure convergence of the scheme, it does not seem enough to just combine the computations of \cite{CDS} with the classical estimates on conformal mappings. In particular in order to close the argument we impose a much faster rate of convergence for $g-u_q^\sharp e$: in \cite{CDS} it was sufficient to choose exponentially decaying $\delta_q$ (and exponentially growing $\lambda_q$), whereas in this note we take advantage of a double exponential Ansatz. This idea is in fact borrowed from \cite{DS3}, where a scheme with a double exponential decay was used to produce H\"older solutions to the Euler equations.

\medskip

The rest of the paper is organized as follows.

Section \ref{s:preliminaries} collects the technical preliminary lemmas and propositions which will be used in the proofs of Proposition \ref{p:main} and Theorem \ref{t:main}.

The proof of Proposition \ref{p:main} is split into the Sections \ref{s:iter_start}, \ref{s:est_1}, \ref{s:est_2} and \ref{s:iter_conclusion}.  Section \ref{s:iter_start} describes how to reach $u_{q+1}$ from $u_q$ and in particular it gives the precise formulas for the two oscillatory perturbations which we need to add. We will then collect in Section \ref{s:est_1} the estimates concerning the first perturbation and in Section \ref{s:est_2} the ones concerning the second perturbation. Section \ref{s:iter_conclusion} will finally conclude the proof of Proposition \ref{p:main}. 

Section \ref{s:main} will prove Theorem \ref{t:main} using Proposition \ref{p:main}. In fact the proof is not completely straightforward since we have to show the existence of a map $u_0$ which is $C^0$ close to the map $\bar u$ of Theorem \ref{t:main} and at the same time satisfies the requirements of Proposition \ref{p:main} (with $q=0$), in order to be able to start the iterative procedure. 
In Section \ref{s:main2} we give briefly the necessary technical modifications to prove Theorem \ref{t:main2}.

One key technical point is Proposition \ref{p:conformal}, which addresses rather well-known regularity properties of conformal changes of coordinates. However, it is crucial for us to have an explicit (linear) dependence of certain H\"older norms of the change of coordinates in terms of corresponding norms of the metric. Since we have not been able to find the relevant statements in the literature, we have included a proof of Proposition \ref{p:conformal} in the Appendix.

\subsection{Acknowledgments} The research of Camillo De Lellis and Dominik Inauen has been supported by the grant $200021\_159403$ of the Swiss National Foundation. L\'aszl\'o Sz\'ekelyhidi acknowledges the support of the ERC Grant Agreement No. 277993.

\section{Preliminaries}\label{s:preliminaries}

\subsection{H\"older spaces} In the following $m\in \N$, $\alpha\in (0,1)$, and $\beta$ is a multi-index. Moreover we will always assume that the domain of definition of any map is a disk $\overline{D}_r\subset \R^2$ with radius $r\in [1,2]$. The maps $f$ can be real-valued, vector-valued, matrix-valued or generally tensor-valued. In all these cases we endow the targets with the standard Euclidean norms, for which we will use the notation $|f (x)|$. 
We introduce the usual H\"older norms as follows.
First of all, the supremum norm is denoted by $\|f\|_0:=\sup |f|$.  We define the H\"older seminorms 
as
\begin{equation*}
\begin{split}
[f]_{m}&=\max_{|\beta|=m}\|D^{\beta}f\|_0\, ,\\
[f]_{m+\alpha} &= \max_{|\beta|=m}\sup_{x\neq y}\frac{|D^{\beta}f(x)-D^{\beta}f(y)|}{|x-y|^{\alpha}}\, .
\end{split}
\end{equation*}
The H\"older norms are then given by
\begin{eqnarray*}
\|f\|_{m}&=&\sum_{j=0}^m[f]_j\, ,\\
\|f\|_{m+\alpha}&=&\|f\|_m+[f]_{m+\alpha}.
\end{eqnarray*}
We then recall the standard ``Leibniz rule'' to estimate norms of products
\begin{equation}\label{e:Holderproduct}
[fg]_{r}\leq C\bigl([f]_r\|g\|_0+\|f\|_0[g]_r\bigr) \qquad \mbox{for any $1\geq r\geq 0$}
\end{equation}
and the usual interpolation inequalities
\begin{equation}\label{e:Holderinterpolation2}
[f]_{s}\leq C\|f\|_0^{1-\frac{s}{r}}[f]_{r}^{\frac{s}{r}} \qquad \mbox{for all $r\geq s\geq 0$}.
\end{equation}
The following version of estimate \eqref{e:Holderinterpolation2}, with explicit constant, will be useful at a certain stage:
\begin{equation}\label{e:interpolation_explicit}
\|f\|_\alpha \leq \|f\|_0 + 2 \|f\|_0^{1-\alpha}\|Df\|_0^\alpha \qquad \mbox{for all $0\leq \alpha \leq 1$.}
\end{equation}

We also collect two classical estimates on the H\"older norms of compositions. These are also standard, for instance
in applications of the Nash-Moser iteration technique.

\begin{proposition}\label{p:chain}
Let $0\leq\alpha<1$, $\Psi: \Omega \to \mathbb R$ and $u: \R^n\supset U \to \Omega$ be two $C^{m,\alpha}$ functions, with $\Omega\subset \R^N$. 
Then there is a constant $C$ (depending only on $\alpha$, $m$,
$\Omega$ and $U$) such that
\begin{align}
\left[\Psi\circ u\right]_{m+\alpha}&\leq C[u]_{m+\alpha}\left( [\Psi]_1+\|u\|_0^{m-1}[\Psi]_m\right)\nonumber\\
				      &+C[\Psi]_{m+\alpha}\left(\|u\|_0^{m-1}[u]_m\right)^{\frac{m+\alpha}{m}}\label{e:chain3}\, ,\\
\left[\Psi\circ u\right]_{m+\alpha} &\leq C\left([u]_{m+\alpha}[\Psi]_1+[u]_1^{m+\alpha}[\Psi]_{m+\alpha}\right)\label{e:chain4}\, .
\end{align} 
Let $f, g: \R^n\supset U \to \R$ two $C^{m,\alpha}$ functions. Then there is a constant $C$ (depending only on $\alpha$, $m$, $n$ and $U$) such that
\begin{equation}\label{e:product}
[fg]_{m+\alpha} \leq C(\|f\|_0 [g]_{m+\alpha} + \|g\|_0 [f]_{m+\alpha})\, .
\end{equation}
\end{proposition}
\begin{proof}
The chain rule can be written as 
\begin{equation}\label{e:chain2}
    D^m\left(\Psi \circ u \right)= \sum\limits_{i=1}^m\left(D^i\Psi\circ u\right)\sum\limits_{k} C_{i,k} \left(Du \right)^{k_1}\cdot\dots\cdot\left( D^mu \right)^{k_m}\, ,
     \end{equation}
  where $C_{i,k}$ are constants and $k=\left( k_1,\dots,k_m \right)$ is a multi-index with 
    \begin{equation*}
      \sum k_j=i, \quad \sum j k_j = m\,.
    \end{equation*}
The claim then follows by the Leibniz rule \eqref{e:Holderproduct} and a repeated application of the interpolation inequalities \eqref{e:Holderinterpolation2} to \eqref{e:chain2}. Statement \eqref{e:product} is a 
straightforward consequence of the usual Leibniz rule, interpolation and the Young inequality.
\end{proof}
\begin{remark}
 Observe that if $\alpha=0$ we have the estimates
  \begin{align}
   [\Psi\circ u]_m &\leq C[u]_m\left([\Psi]_1+\|u\|_0^{m-1}[\Psi]_m  \right) \label{e:chain0},\\
   [\Psi\circ u]_m &\leq C\left( [u]_m[\Psi]_1+[u]_1^{m}[\Psi]_m \right) \label{e:chain1}.
  \end{align}
\end{remark}

\subsection{Quadratic mollification estimate} 
We will often use regularizations of maps $f$ by convolution with a standard mollifier $\varphi_\ell (y) := \ell^{-2} \varphi (\frac{y}{\ell})$, where $\varphi\in C^\infty_c (D_1)$ is assumed to have integral $1$ and to be non negative and rotationally symmetric. 
Since however the domain of $f$ will be $D_r$ (resp. $\overline{D}_r$), we fix the convention that the convolution $f*\varphi_\ell$ is defined in $D_{r-\ell}$ (resp. $\overline{D}_{r-\ell}$).  

\begin{lemma}\label{l:mollify}
For any $r,s\geq 0$ and $0<\alpha\leq 1$ we have
\begin{align}
&[f*\varphi_\ell]_{r+s}\leq C\ell^{-s}[f]_r,\label{e:mollify1}\\
&[f-f*\varphi_\ell]_r \leq C\ell^{2}[f]_{2+r},\label{e:mollify2}\\
&\|f-f*\varphi_\ell\|_r \leq C\ell^{2-r}[f]_{2}, \qquad \mbox{if $0\leq r\leq2$} \label{e:mollify4}\\
&\|(fg)*\varphi_\ell-(f*\varphi_\ell)(g*\varphi_\ell)\|_r\leq C\ell^{2\alpha -r}\|f\|_\alpha\|g\|_\alpha\, ,\label{e:mollify3}
\end{align}
where the constants $C$ depend only upon $s$, $r$, $\alpha$ and $\varphi$. 
\end{lemma}
\begin{proof}
 Except for \eqref{e:mollify4}, the other estimates are contained in \cite[Lemma 1]{CDS}. The additional claim \eqref{e:mollify4} can be seen as follows. Recall the estimate
 \[
  \|f-f*\varphi_\ell\|_0 \leq C \ell [f]_1\, ,
 \]
which can be derived using the mean value theorem and an integration. 
We combine this estimate with \eqref{e:Holderinterpolation2} and \eqref{e:mollify2} to get 
\begin{align*}
 [f-f*\varphi_\ell]_r &\leq C\|f-f*\varphi_\ell\|_0^{1-r}[f-f*\varphi_\ell]_1^{r}\\
 \leq & C \left( \ell^{2}\|D^2f\|_0 \right)^{1-r}\left( \ell\|D^2f\|_0 \right)^{r}\leq C\ell^{2-r}[f]_2\, ,
\end{align*}
whenever $0\leq r\leq1$. If however $1\leq r\leq2$, we invoke the trivial inequality \[ [f-f*\varphi_\ell]_2\leq C[f]_2\] to deduce
\begin{equation*}
 [f-f*\varphi_\ell]_r \leq C\|\nabla f-\nabla f*\varphi_\ell\|_0^{2-r}[\nabla f-\nabla f*\varphi_\ell]_1^{r-1} \leq C \ell^{2-r}[f]_2\, ,
\end{equation*}
from which the claim follows.
\end{proof}

\subsection{Conformal coordinates}\label{s:conformal}
A crucial ingredient of our proof is the following proposition on the existence of conformal coordinates for $C^{N, \alpha}$ metrics. Although such existence is a very classical fact, we need an explicit dependence of the norms of the coordinates in terms of the regularity of the metric. Since we have not been able to find a precise reference in the literature, we include a proof in the appendix.

\begin{proposition}\label{p:conformal}
  For any $N,\alpha,\beta$ with $N\in \N,N\geq 1$, $0<\beta\leq \alpha <1$ there exist constants $C  (N, \alpha, \beta), \sigma_1 (N, \alpha, \beta)>0 $ and $\bar{C} (\alpha)$ such that the following holds. If $1\leq r \leq 2$ and $g$ is a $C^{N,\alpha}$ metric on $\overline{D}_r$ with 
\begin{equation}\label{e:smallness}
\|g-e\|_{\alpha} \leq \sigma_1
\end{equation}
then there exists a coordinate change $\Phi: \overline{D}_r \to \R^2$ and a function $\rho: \overline{D}_r \rightarrow \R^+$ satisfying 
\begin{equation}\label{e:conformal}
g=\rho^2\left( \nabla\Phi_1\otimes\nabla\Phi_1+ \nabla\Phi_2\otimes\nabla\Phi_2\right)
\end{equation}
and the following estimates
\begin{align}
&\|\rho-1\|_\alpha+\|D \Phi-Id\|_\alpha \leq \bar C\|g-e\|_\alpha \label{e:conf_0}\\
&\|D^k\rho\|_\beta +\|D^{k+1}\Phi\|_\beta \leq C\|g-e\|_{k+\beta}\qquad \forall 1 \leq k \leq N\, .\label{e:conf_k}
\end{align}
\end{proposition}

\subsection{Oscillatory functions} The construction of $u_{q+1}$ is based on adding to the map $u_q$ suitable ``wrinkles'', namely suitable perturbations.
The basic model for this perturbation takes advantage of a pair of real-valued functions with very specific properties, which we will detail here.

\begin{proposition}\label{p:Kuiper}
There exists $\delta_\star>0$ and a function $\Gamma = (\Gamma^t, \Gamma^n) \in C^\infty ([0, \delta_\star]\times \R, \R^2$) with the following properties
\begin{itemize}
\item[(a)] $\Gamma (s, \xi) = \Gamma (s, \xi + 2\pi)$ for every $s, \xi$;
\item[(b)] $(1+ \partial_\xi \Gamma^t)^2 + (\partial_\xi \Gamma^n)^2 = 1+s^2$;
\item[(c)] The following estimates hold:
\begin{align}
\|\partial^k_\xi \Gamma^n (s, \cdot)\|_0 \leq &C (k) s \label{e:Gamma_n}\\
\|\partial^k_\xi \Gamma^t (s, \cdot)\|_0 \leq & C(k) s^2 \label{e:Gamma_t}\\
\|\partial_s \partial^k_\xi \Gamma^t (s, \cdot)\|_0 \leq & C(k) s\, .\label{e:Ds_Gamma_t}
\end{align}
\end{itemize}
\end{proposition} 
\begin{proof}
Except for \eqref{e:Gamma_t} the remaining claims are contained in \cite[Lemma 2]{CDS}. The idea is to let $\Gamma$ have the form 
 \[ \Gamma(s,\xi) \coloneqq \int\limits_0^{\xi}\left( \sqrt{1+s^2}\left(\cos(f(s)\sin(\tau)), \sin(f(s)\sin(\tau))\right)-(1,0) \right)d\tau\, ,\]
 for an appropriately chosen function $f$ such that (a), \eqref{e:Gamma_n} and \eqref{e:Ds_Gamma_t} are fulfilled. (b) is satisfied by construction. The additional statement \eqref{e:Gamma_t} follows from integrating \eqref{e:Ds_Gamma_t} in $s$.
\end{proof}

\section{Proof of Proposition \ref{p:main}, Part I}\label{s:iter_start}

\subsection{Hierarchy of parameters} A first ingredient in the construction of $u_{q+1}$ is to smooth $u_q$ suitably via a standard mollification. For this we introduce the mollification parameter $\ell$, which is rather small: indeed it is defined by the relation
\begin{equation}\label{e:ell}
\ell^{2-\alpha} := \frac{1}{\tilde C} \frac{\delta_{q+1}}{\delta_q \lambda_q^2}\, ,
\end{equation}
where $\tilde C$ is a constant larger than $1$ which depends only upon $\alpha$, $g$, $\sigma_0$ and $\bar C$ and which will be specified in Section \ref{s:regularize} below. 

The map $u_{q+1}$ will be obtained from (a suitable regularization of) the map $u_q$ in two steps. First we will add an oscillatory perturbation whose frequency is
\begin{equation}\label{e:def_mu}
\mu := \hat{C} \frac{\delta_{q+1} \lambda_{q+1}^\alpha}{\delta_{q+2} \ell}\, ,
\end{equation}
where the constant $\hat{C}$, larger than $1$, depends only upon $\alpha$, $g$, and $\sigma_0$ (we specify its choice in Section \ref{s:iter_conclusion}).
We will then choose a second perturbation whose frequency is $\lambda_{q+1}$. 

We next record a few inequalities among the parameters which will be rather useful in simplifying some of our estimates in the remaining sections. Except for the very first inequality in \eqref{e:mu}, which requires a choice of $a$ sufficiently large compared to the constant $\hat{C}$, all the others are immediate from the restrictions imposed so far on all the various parameters.
\begin{align}
&\delta_q \lambda_q^2  \geq 1\, ,\label{e:deltalambda>1}\\
& \lambda_{q+1} \geq  \mu \geq \ell^{-1}\geq \lambda_q\, ,\label{e:mu} \\
& \delta_q^{\sfrac{1}{2}} \lambda_q\leq \delta_q^{\sfrac{1}{2}} \lambda_q \ell^{-\sfrac{\alpha}{2}}\leq\delta_{q+1}^{\sfrac{1}{2}} \ell^{-1} \leq \delta_{q+1}^{\sfrac{1}{2}} \mu \leq \delta_{q+1}^{\sfrac{1}{2}} \lambda_{q+1}\, , \label{e:ell_lambda}
\end{align}
The first inequality \eqref{e:deltalambda>1} follows from
$\delta_q \lambda_q^2 = a^{c^2b^{2q+2} - b^q}\geq a^{b^2-1}$ (where we have used $c, b >1$). Observe that this easily implies $\ell \leq 1$ (recall that $\delta_{q+2}$ and $\tilde{C}^{-1}$ are both smaller than $1$), which in turn gives the first inequality in \eqref{e:ell_lambda}. Note also that
the last inequality in \eqref{e:mu} is weaker than the second inequality in \eqref{e:ell_lambda}:
\[
\ell^{-1} \geq \ell^{-1+ \sfrac{\alpha}{2}} \geq \frac{\delta_q^{\sfrac{1}{2}}}{\delta_{q+1}^{\sfrac{1}{2}}} \lambda_q \geq \lambda_q\, .
\]
Coming to the second inequality in \eqref{e:ell_lambda}, observe that, by the definition of $\ell$, this is just the requirement that $\tilde{C}\geq 1$. 
As for the last two inequalities in \eqref{e:ell_lambda} are equivalent to the first two in \eqref{e:mu}, which will be shown below.
Moreover, since $\hat{C} > 1$, $\lambda_{q+1} >1$ and $\delta_{q+1} \geq \delta_{q+2}$, the second inequality in \eqref{e:mu} is obvious. 

We are therefore left with showing the first inequality in \eqref{e:mu} which, as already mentioned, needs a sufficiently large $a$. As it can be readily checked from the definition of $\mu$, such inequality is in fact equivalent to $\delta_{q+2} \lambda_{q+1}^{1-\alpha} \geq \hat{C} \delta_{q+1} \ell^{-1}$. But we record in fact a much stronger inequality, which turns out to be the key relation to conclude the estimates in Proposition \ref{p:main}, as it will become apparent in Section \ref{s:iter_conclusion}.
More precisely, given any constant $\underline{C}$ which depends upon $\alpha, g, \sigma_0$ and $\bar{C}$, the following inequality holds provided $a$ is chosen large enough:
\begin{equation}\label{e:key_parameters}
\delta_{q+2}^2 \lambda_{q+1}^{1-2\alpha} \geq \underline{C} \delta_{q+1}^2 \ell^{-1} \, .
\end{equation}
In fact such inequality is equivalent to
\[
\delta_{q+2}^2 \lambda_{q+1}^{1-2\alpha}
\geq \underline{C} \tilde{C}^{1/(2-\alpha)} \delta_{q+1}^{2-1/(2-\alpha)} \delta_q^{1/(2-\alpha)} \lambda_q^{2/(2-\alpha)}\, .
\]
Taking the logarithm in base $a$ this is equivalent to
\[
 (c (1-2\alpha)-2) b^{q+2} \geq \left(\frac{1+2c}{2-\alpha}-2\right) b^{q+1} - \frac{1}{2-\alpha} b^q + \log_a \underline{C} + \frac{1}{2-\alpha} \log_a \tilde{C}\, .
\]
The latter follows for a sufficiently large $a$ (depending upon $b$, $c$, $\tilde{C}$ and $\underline{C}$) provided
\[
\big(c(1-2\alpha) -2\big) b^2 > \left(\frac{1+2c}{2-\alpha}-2\right) b - \frac{1}{2-\alpha}\, ,
\]
which is equivalent to
\[
c b ((2-\alpha)(1-2\alpha) b -2 ) > 2 (2-\alpha) b^2 + (1-2(2-\alpha)) b -1\, .
\]
The latter inequality is however obviously implied by \eqref{e:abc_1} and \eqref{e:abc_2}.

\subsection{Constants} In the rest of the paper we will deal with several estimates where we bound norms of various functions using the parameters introduced so far, namely $\delta_q, \lambda_q, \ell, \mu$ and $\lambda_{q+1}$. In front of the expressions involving such parameters there will always be some constants, independent of $a$, $b$ and $c$. However it is important to distinguish between two types of such constants: the ones which depend only upon $\alpha$, $g$ and $\sigma_0$ will be denoted by $C$, whereas the ones which depend also upon the $\bar{C}$ of Proposition \ref{p:main} will be denoted by $C^\star$. Note also that the parameter $\sigma_0$ will in fact be chosen as a function of $\alpha$ in Section \ref{s:conformal_change}. Therefore the constants denoted by $C$ will depend only upon $\alpha$ and $g$, whereas those denoted by $C^\star$ will depend, additionally, also upon $\bar{C}$. {Moreover, the values of $C$ and $C^\star$ may change from line to line.}

\subsection{Regularization}\label{s:regularize} 
 Having fixed a standard mollifier $\varphi$, we then define
\begin{equation}\label{e:tau}
h_q := \frac{g*\varphi_\ell - (u_q*\varphi_\ell)^\sharp e}{\delta_{q+1}}-\frac{\delta_{q+2}}{\delta_{q+1}} e\, . 
\end{equation}
Observe that 
\[
(u_q*\varphi_\ell)^\sharp e +\delta_{q+1} h_q = g*\varphi_\ell - \delta_{q+2} e = g_{q+1} + (g*\varphi_\ell-g)\, .
\] 
So the strategy of the proof will
be to perturb $u_q*\varphi_\ell$ to a map $u_{q+1}$ such that 
\[
u_{q+1}^\sharp e = (u_q*\varphi_\ell)^\sharp e +\delta_{q+1} h_q + E = g_{q+1} + E + (g*\varphi_\ell-g) \, ,
\] 
(cf. \ref{e:metric_error})
where the error term $E$ is suitably small. Before coming to the construction of the map $u_{q+1}$ we deal in this section with the smallness conditions to be imposed on $\ell$.

First of all, by choosing $\tilde{C}$ larger than a geometric constant and $a$ sufficiently large (depending upon $b$ and $c$), we can assume 
that $\ell \leq 2^{-{q-2}}$, so that $h_q$ is in fact defined on $\overline{D}_{1+ 2^{-q-2}}$.
Next, using Lemma \ref{l:mollify} we can estimate 
\begin{align*}
& \|h_q - e\|_\alpha \leq \;  \frac{\delta_{q+2}}{\delta_{q+1}} + \frac{1}{\delta_{q+1}}\| g*\varphi_\ell - (u_q*\varphi_\ell)^\sharp e -\delta_{q+1}e\|_\alpha\\
 \leq &\; a^{-(b-1)} + \frac{1}{\delta_{q+1}} \big(\|(u_q^\sharp e) * \varphi_\ell - (u_q*\varphi_\ell)^\sharp e\|_\alpha
+ \|(g_q - u_q^\sharp e)* \varphi_\ell\|_\alpha\\
& \qquad\qquad\qquad\qquad\qquad\qquad\qquad +  \|g - g*\varphi_\ell\|_\alpha\big)\nonumber\\
\leq & \sigma_0 + C^\star \frac{\ell^{2-\alpha} \delta_q \lambda_q^2}{\delta_{q+1}} + \sigma_0 + \frac{C}{\delta_{q+1}} \|D^2 g\|_0 \ell^{2-\alpha}\\
\stackrel{\eqref{e:deltalambda>1}}{\leq} & 2\sigma_0 + C^\star \frac{\ell^{2-\alpha} \delta_q \lambda_q^2}{\delta_{q+1}}\leq 3\sigma_0\, ,
\end{align*}
where the latter inequality specifies the condition needed on $\tilde{C}$ in \eqref{e:ell}.

Similarly, for $1\leq k \leq 4$, we can bound
\begin{align}
\|D^k h_q\|_0 \leq& \frac{1}{\delta_{q+1}}\Big(\|D^k (g - u_q^\sharp e)*\varphi_\ell\|_0\nonumber\\
 &\qquad\qquad+ \|D^k ((u^\sharp e)*\varphi_\ell - (u_q *\varphi_\ell)^\sharp e)\|_0 \Big)\nonumber\\
&\leq C\ell^{\alpha-k}\sigma_0 + C^\star \frac{\delta_q\lambda_q^2}{\delta_{q+1}} \ell^{2-k}\leq C \ell^{\alpha -k}\, ,
\end{align}
where we have used \eqref{e:ell} and Lemma \ref{l:mollify}.
Interpolating, for any $0\leq k\leq 3$ we then get
\begin{equation}\label{e:tau_k_1}
\|h_q - e\|_{k + \alpha} \leq C \ell^{-k}\, .
\end{equation}
We summarize the conclusions of the previous paragraphs in the following lemma.

\begin{lemma}\label{l:convolution}
If we choose $\tilde{C}$ sufficiently large, depending upon $\alpha, g$ and $\bar{C}$, we then have
\begin{align}
&\|h_q-e\|_\alpha \leq 3 \sigma_0\label{e:3sigma}\\
&\|h_q-e\|_{k+\alpha} \leq C \ell^{-k}\qquad \mbox{for $1\leq k \leq 3$,}\label{e:tau_k}
\end{align}
where the constant $C$ depends only upon $\alpha$ and $g$.
\end{lemma}

\subsection{Conformal diffeomorphism}\label{s:conformal_change}
We now wish to apply Proposition \ref{p:conformal} with $\beta = \alpha >0$ and $N=3$. This requires to choose $\sigma_0$ such that $3\sigma_0 \leq \sigma_1$, where $\sigma_1$ is the constant appearing in \eqref{e:smallness}. We thus find maps $\Phi$ and $\rho$ such that
\[ 
h_q = \rho^2\left( \nabla\Phi_1\otimes \nabla \Phi_1+\nabla\Phi_2\otimes \nabla \Phi_2 \right)\, .
\]
Furthermore, dividing $\rho$ by $\max \rho$, multiplying $\Phi$ by $\max \rho$ and using \eqref{e:conf_0}, 
we can assume that
\begin{equation}\label{e:rho_DPhi_0}
\frac{1}{2} \leq \rho \leq 2 \qquad \|D \Phi- {\rm Id}\|_0 \leq \frac{1}{2}\, ,
\end{equation}
provided $\sigma_0$ is chosen sufficiently small. This exhausts the condition on $\sigma_0$: note that they depend only upon $\alpha$, since $N$ and $\beta$ in Proposition \ref{p:conformal} are fixed to be $3$ and $\alpha$.

Moreover, for any $1\leq k \leq 3$ we apply \eqref{e:conf_k}
and \eqref{e:tau_k} to estimate
\begin{equation}\label{e:rho_DPhi_k}
\|D^k \rho\|_\alpha + \|D^{k+1} \Phi\|_\alpha \leq C \ell^{-k} \,.
\end{equation}

\subsection{Adding the first {primitive metric}} We next set $w:= u_q *\varphi_\ell$ and we define the following two three-dimensional vectors:
\begin{equation}\label{e:tau_1}
\tau_1 := Dw (Dw^T Dw)^{-1} \nabla \Phi_1\, 
\end{equation}
and
\begin{equation}\label{e:nu_1}
\nu_1 := \frac{\partial_{x_1} w \times \partial_{x_2} w}{|\partial_{x_1} w \times \partial_{x_2} w|}\, . 
\end{equation}
Observe that $\nu_1$ is in the kernel of $Dw^T$ (or, in other words, $\nu_1 (x)$ is a unit normal to the tangent plane $T_{w (x)} ({\rm Im}\, (w))$). Hence it follows easily that $\tau_1$ and $\nu_1$ are orthogonal.

We next normalize these vectors suitably, defining
\begin{align}
& t_1 := \frac{\tau_1}{|\tau_1|^2}\label{e:t_1}\, ,\\
& n_1 := \frac{\nu_1}{|\tau_1|}\, .\label{e:n_1}
\end{align}
Finally, we define the first perturbation of $w$, namely the map $v$ given by the formula
\begin{equation}\label{e:v}
v = w + \frac{1}{\mu} \Gamma^t \left( \delta_{q+1}^{\sfrac{1}{2}} |\tau_1| \rho, \mu \Phi_1\right) t_1 + \frac{1}{\mu} \Gamma^n 
\left(\delta_{q+1}^{\sfrac{1}{2}} |\tau_1|\rho, \mu \Phi_1\right)\, n_1\, ,
\end{equation}
whereas we define
\begin{equation}\label{e:E1}
E_1 := v^\sharp e - (w^\sharp e + \delta_{q+1} \rho^2 \nabla \Phi_1 \otimes \nabla \Phi_1)\, .
\end{equation}

\subsection{Adding the second {primitive metric}} The map $u_{q+1}$ is then obtained by adding a similar second perturbation to the map $v$. More precisely we define this time
\begin{align}
\tau_2 &:= Dv (Dv^T Dv)^{-1} \nabla \Phi_2\, ,\label{e:tau_2}\\
\nu_2 &:= \frac{\partial_{x_1} v \times \partial_{x_2} v}{|\partial_{x_1} v\times \partial_{x_2} v|}\,, \label{e:nu_2}\\
t_2& := \frac{\tau_2}{|\tau_2|^2}\, ,\label{e:t_2}\\
n_2 & := \frac{\nu_2}{|\tau_2|}\, .\label{e:n_2}
\end{align}
The map $u_{q+1}$ is then given by the following formula (analogous to \eqref{e:v}):
\begin{equation}\label{e:u_q+1}
u_{q+1}= v + \frac{1}{\lambda_{q+1}} \Gamma^t \left( \delta_{q+1}^{\sfrac{1}{2}} |\tau_2| \rho, \lambda_{q+1} \Phi_2\right) t_2 + \frac{1}{\lambda_{q+1}} \Gamma^n 
\left(\delta_{q+1}^{\sfrac{1}{2}} |\tau_2|\rho, \lambda_{q+1} \Phi_2\right)\, n_2\, .
\end{equation}
Similarly we define
\begin{equation}\label{e:E2}
E_2 := u_{q+1}^\sharp e - (v^\sharp e + \delta_{q+1} \rho^2 \nabla \Phi_2 \otimes \nabla \Phi_2)\, .
\end{equation}
Observe that we have the following identity:
\begin{align}
E :=&  E_1 + E_2 = u_{q+1}^\sharp e - (w^\sharp e + \delta_{q+1} \rho^2 (\nabla \Phi_1 \otimes \nabla \Phi_1 + \nabla \Phi_2 \otimes \nabla \Phi_2))\nonumber\\
=&\,  u_{q+1}^\sharp e - w^\sharp e - \delta_{q+1} h_q  = u_{q+1}^\sharp e + \delta_{q+2} e - g*\varphi_\ell \nonumber \\ =&\, u_{q+1}^\sharp e - g_{q+1}+(g-g*\varphi_\ell)\, .\label{e:metric_error}
\end{align}
Hence \begin{align} 
       \|g_{q+1}-u_{q+1}^{\sharp}e\|_0 &\leq \|E\|_0+\|g-g*\varphi_\ell\|_0\, ,\\
       \|D(g_{q+1}-u_{q+1}^{\sharp}e)\|_0 &\leq \|D E\|_0 + \|D(g-g*\varphi_\ell)\|_0\, .
      \end{align}
For $\alpha$ sufficiently small and $a$ sufficiently big one can achieve
\begin{align}
 \|g-g*\varphi_\ell\|_0 &\leq C\|D^{2}g\|_0\ell^{2} \leq \frac{\sigma_0}{6}\delta_{q+2}\lambda_{q+1}^{-\alpha} \, ,\label{e:gmoll1}\\
 \|D(g-g*\varphi_\ell)\|_0 &\leq C\|D^{2}g\|_0 \ell \leq \frac{\sigma_0}{6}\delta_{q+2}\lambda_{q+1}^{1-\alpha} \, . \label{e:gmoll2}
\end{align}
To see this, note that \eqref{e:gmoll1} is implied by the condition
\[C^\star \frac{\delta_{q+1}}{\delta_q\lambda_q^{2}} \leq \delta_{q+2}\lambda_{q+1}^{-\alpha}\, ,\]
which for $a(\bar C)$ big enough is guaranteed if 
\[b^{2}-b+1 <(2-\alpha b)cb\, ,\]
or equivalently
\begin{equation}\label{e:csupp1}
c> \frac{b^{2}-b+1}{b(2-\alpha b)}\, .
\end{equation}
Similarly \eqref{e:gmoll2} follows if 
\[ C^{\star}\frac{\delta_{q+1}^{\sfrac{1}{2}}}{\delta_q^{\sfrac{1}{2}}\lambda_q} \leq \delta_{q+2}\lambda_{q+1}^{1-\alpha}\, ,\]
which (for $a(\bar C)$ big enough) is satisfied whenever 
\begin{equation}\label{e:csupp2}
c > \frac{2b^{2}-b+1}{2b(1+(1-\alpha)b)}\, .
\end{equation}
 Now for any $\alpha >0$, $b>1$ which satisfy the bounds of the Proposition we have
\[\frac{b^{2}-b+1}{b(2-\alpha b)}> \frac{2b^{2}-b+1}{2b(1+(1-\alpha)b)}\, .\]
 Indeed, since $b < \frac{3}{2}$ and $\alpha < \alpha_0$, provided $\alpha_0$ is small enough both denominators
in the fractions above are positive. Hence the inequality
is equivalent to \[ 2b^{2}+(\alpha-4)b+(2-\alpha) = (b-1)(\alpha +2b-2) > 0\, ,\] which for $b>1$ and $\alpha >0$ is always true.  
Hence \eqref{e:csupp1} implies \eqref{e:csupp2}. 

 Next, observe that the left hand side of \eqref{e:abc_2} is larger than
$g_\alpha (b) = \frac{(4-2\alpha)b +1}{2b}$, so \eqref{e:abc_2} implies $c > g_\alpha (b)$.  The bound \eqref{e:csupp1} is
instead $c > h_\alpha (b) = \frac{b^{2}-b+1}{b(2-\alpha b)}$. On the other hand on the interval $[1, \frac{3}{2}]$, $g_\alpha$ and $h_\alpha$ converge uniformly, as $\alpha \downarrow 0$, to the functions $ g_0 (b) = 2 + \frac{1}{2b}$ and $h_0 (b) = \frac{b^2-b+1}{2b}$. Since on $[1, \frac{3}{2}]$ $g_0$ is strictly larger than $h_0$, we infer that  
for $\alpha$ small \eqref{e:abc_2} guarantees \eqref{e:csupp1}. In particular we conclude that for $a$ big enough \eqref{e:abc_2} guarantees \eqref{e:gmoll1}
and \eqref{e:gmoll2}.\\
Thus, the goal of most of the remaining sections is to prove that the desired bounds hold for $\|E\|_0$, $\|D E\|_0$, $\|u_{q+1} - u_q\|_0$, $\|D (u_{q+1} - u_q)\|_0$ and $\|D^2 u_{q+1}\|_0$.

\section{Estimates on $v$ and $E_1$}\label{s:est_1}
Our goal in this subsection is to estimate the $C^0$ norms of $v-u_q$, $D^k v$, $E_1$ and $D E_1$.
To this aim we introduce the functions
\begin{align}
A_1^t & := \partial_\xi \Gamma^t \left( \delta_{q+1}^{\sfrac{1}{2}} |\tau_1| \rho, \mu \Phi_1\right)\, ,\\
A_1^n& := \partial_\xi \Gamma^n \left( \delta_{q+1}^{\sfrac{1}{2}} |\tau_1| \rho, \mu \Phi_1\right)\, ,\\
B_1^t& := \partial_s \Gamma^t \left( \delta_{q+1}^{\sfrac{1}{2}} |\tau_1| \rho, \mu \Phi_1\right)\, ,\\
B_1^n& := \partial_s \Gamma^n \left( \delta_{q+1}^{\sfrac{1}{2}} |\tau_1| \rho, \mu \Phi_1\right)\, ,\\
C_1^t& := \Gamma^t \left( \delta_{q+1}^{\sfrac{1}{2}} |\tau_1| \rho, \mu \Phi_1\right)\, ,\\
C_1^n& := \Gamma^n \left( \delta_{q+1}^{\sfrac{1}{2}} |\tau_1| \rho, \mu \Phi_1\right)\, ,
\end{align}
and we decompose the derivative of $v$ as 
\begin{align}
&Dv =  Dw + \underbrace{A_1^t\,  t_1 \otimes \nabla \Phi_1 + A_1^n\, n_1 \otimes \nabla \Phi_1}_{=:\bA_1}\nonumber\\
&+ \underbrace{\frac{\delta^{\sfrac{1}{2}}_{q+1}}{\mu} (B^t_1\, t_1 + B^n_1\, n_1) \otimes (\rho \nabla |\tau_1| + |\tau_1| \nabla \rho)}_{=:\bB_1}
+ \underbrace{\frac{1}{\mu} \left(C_1^t\, Dt_1 + C_1^n Dn_1\right)}_{=:\bC_1}\, .
\end{align}

\subsection{First technical lemma} In the next lemma we collect the estimates of the $C^0$ norm of the derivatives of the
various quantities introduced above.

\begin{lemma}\label{l:ugly_lemma_1}
Let $\tilde{C}$ be fixed so that Lemma \ref{l:convolution} holds and $\hat{C} \geq 1$. If $a \geq a_0 (\alpha, g, b,c, \bar{C})$ for some $a_0$ sufficiently large, then
there are constants $C$ (depending upon $\alpha$ and $g$ but not on $\bar{C}$) such that
\begin{equation}\label{e:tau_nu_0}
{C^{-1} \leq |\tau_1| \leq C}
\end{equation}
and
\begin{align}
\|w-u_q\|_0 &\leq C \delta_{q+1}^{\sfrac{1}{2}}\ell\label{e:w_0}\, ,\\
\|D (w-u_q)\|_0 & \leq C \delta_{q+1}^{\sfrac{1}{2}}\label{e:w_1}\, ,\\
\|Dw\|_0 &\leq C\label{e:w_1_bis}\, ,\\
\|D^{k} w\|_0&\leq C \delta_{q+1}^{\sfrac{1}{2}}\ell^{1-k}\qquad \ \ \ \, \mbox{for $2 \leq k \leq 4$,}\label{e:w_k}\\
\|D^k \nu_1\|_0&\leq C \delta_{q+1}^{\sfrac{1}{2}} \ell^{-k}\qquad\ \quad \; \mbox{for $1 \leq k \leq 3$,}\label{e:nu_k}\\
\|D^{k} t_1\|_0 + \|D^{k} \tau_1\|_0 + \|D^{k} n_1\|_0 &\leq C \ell^{-k} \qquad\qquad\quad\;\,\mbox{for $0\leq k\leq 3$,}\label{e:tn_k}\\
\|D^k A_1^t\|_0 + \|D^k C_1^t\|_0 &\leq C \delta_{q+1} \mu^k\qquad\qquad \mbox{for $0\leq k \leq 3$,}\label{e:Gamma_quad}\\
\|D^k A_1^n\|_0 + \|D^k B_1^t\|_0 + \|D^k C_1^n\|_0 &\leq C \delta_{q+1}^{\sfrac{1}{2}} \mu^k\qquad\qquad \mbox{for $0\leq k\leq 3$,}\label{e:Gamma_lin}\\
\|D^k B_1^n\|_0 &\leq C \mu^k\qquad\qquad\quad\;\;\; \mbox{for $0\leq k \leq 3$.}\label{e:Gamma_con}
\end{align}
\end{lemma}
\begin{proof}
Since $\|D\Phi - {\rm Id}\|_0 \leq \frac{1}{2}$, we obviously have $\frac{1}{2} \leq |\nabla \Phi_1|\leq 2$. On the other hand  the estimate \eqref{e:3sigma} on $h_q$ of the previous section implies 
\[
g + 5\delta_{q+1} e \geq w^\sharp e \geq g - 5\delta_{q+1} e\, .
\] 
If we assume $a$ sufficiently large (depending only upon $g$, $b$ and $c$), we conclude $2g \geq w^\sharp e \geq \frac{1}{2} g$. Since $w^{\sharp}e = Dw^T Dw$, this implies that 
\[
C|\nabla \Phi_1|\geq |\tau_1|\geq C^{-1} |\nabla \Phi_1|
\] 
for a constant $C$ which depends only upon $g$, hence \eqref{e:tau_nu_0}  follows.

\medskip

{\bf Estimates on $w$.} Observe that 
\begin{align}
&\|w-u_q\|_0\leq C\ell^2\|D^2 u_q\|_0\leq C^\star \ell^2 \delta_q^{\sfrac{1}{2}} \lambda_q\\
&\|D (w-u_q)\|_0 \leq C \ell\|D^2u_q\|_0\leq C^\star \ell \delta_q^{\sfrac{1}{2}} \lambda_q .
\end{align}
If we choose $a\geq a_0(\alpha,b,c,\bar{C})$ big enough such that $\bar{C} \leq \ell^{-\sfrac{\alpha}{2}}$, then \eqref{e:w_0} and \eqref{e:w_1} follow with the help of  \eqref{e:ell_lambda}. Moreover, \eqref{e:w_1} implies \eqref{e:w_1_bis} by \eqref{e:uniform_C1}. Finally, \eqref{e:w_k} is a consequence of \eqref{e:mollify1}, i.e. $\|D^k w\|_0 \leq C\ell^{2-k} \|D^2 u_q\|_0 $, and $\bar{C}\leq \ell^{-\sfrac{\alpha}{2}}$.\\
Next, observe that $C\geq|\partial_{x_1} w \times \partial_{x_2} w|\geq C^{-1}$ (again due to $2g \geq Dw^T Dw \geq \frac{1}{2} g$). Hence \eqref{e:chain0} implies, for $k\geq 1$, 
\begin{align*}
\|D^k \nu_1\|_0&\leq C [D w]_k\|Dw\|_0 \\
&\leq C[Dw]_k\\
&\leq C \delta_{q+1}^{\sfrac{1}{2}} \ell^{-k}\, .
\end{align*} 

\medskip

{\bf Estimates on $\tau_1, t_1$ and $n_1$.} The $C^0$ estimates in \eqref{e:tn_k} are a trivial consequence of \eqref{e:tau_nu_0}. Again by Proposition \ref{p:chain} we get
\begin{align*}
\|D^k \tau_1\|_0 \leq & C \|Dw\|_0 \|D^{k+1}\Phi\|_0 + C\|D\Phi\|_0 \big(\|D^{k+1} w\|_0 + \|D^2 w\|_0^k\big)\\
\leq & C \ell^{-k} + C \delta_{q+1}^{\sfrac{1}{2}}\ell^{-k} \leq C \ell^{-k}\, .
\end{align*}
A second application of Proposition \ref{p:chain} (combined with \eqref{e:tau_nu_0}) gives the estimates
\begin{equation}\label{e:|tau|_k}
\|D^k|\tau_1|\|_0 + \|D^k |\tau_1|^{-1}\|_0 \leq  C \ell^{-k}\, .
\end{equation}
Combining \eqref{e:|tau|_k} and \eqref{e:nu_k}, from \eqref{e:product} we infer
\[
\|D^{k}n_1\|_0 \leq C \delta_{q+1}^{\sfrac{1}{2}}\ell^{-k} + C \ell^{-k} \leq C \ell^{-k}\, .
\]
We argue similarly to conclude $\|D^{k}t_1\|_0 \leq C \ell^{-k}$. 

\medskip

{\bf Remaining estimates.} The cases $k=0$ of \eqref{e:Gamma_quad}, \eqref{e:Gamma_lin} and \eqref{e:Gamma_con} are all simple consequences of Proposition \ref{p:Kuiper} and $\||\tau_1||\rho|\|_0 \leq C$. For the higher derivatives we consider first $C^t_1$. We introduce the function
\[
\Psi (s, \xi) := \delta_{q+1}^{-1} \Gamma^t (\delta_{q+1}^{\sfrac{1}{2}} s, \xi)
\]
and observe that $\|D^i \Psi\|_0 \leq C (i)$ by the estimates in Proposition \ref{p:Kuiper}(c). If we introduce the map $U = (|\tau_1|\rho, \mu \Phi_1)$ we can then write
\[
\|D^k C_1^t\|_0 = \delta_{q+1} \|D^k (\Psi \circ U)\|_0\, .
\]
On the other hand observe that
\begin{align*}
\|D^k U\|_0 &\leq C \ell^{-k} + C \mu \ell^{1-k} \\
&\stackrel{\eqref{e:mu}}{\leq} C \mu \ell^{1-k}\, .
\end{align*}
Hence, using \eqref{e:chain1} we infer
\begin{align*}
\|D^k C_1^t\|_0 &\leq C \delta_{q+1} (\mu \ell^{1-k} + \mu^k) \\
&\leq C \delta_{q+1} \mu^k\, .
\end{align*}
In case of $A_1^t$, $A^1_n$, $B_1^t$, $C_1^n$ and $B_1^n$ we apply the same argument, keeping the map $U$ as defined above, but changing $\Psi$ respectively to
\begin{align*}
\Psi (s, \xi) := & \delta_{q+1}^{-1} \partial_\xi \Gamma^t (\delta_{q+1}^{\sfrac{1}{2}} s, \xi)\\
\Psi (s, \xi) := & \delta_{q+1}^{-\sfrac{1}{2}} \partial_\xi \Gamma^n (\delta_{q+1}^{\sfrac{1}{2}} s, \xi)\\
\Psi (s, \xi) := & \delta_{q+1}^{-\sfrac{1}{2}} \partial_s \Gamma^t (\delta_{q+1}^{\sfrac{1}{2}} s, \xi)\\
\Psi (s, \xi) := & \delta_{q+1}^{-\sfrac{1}{2}} \Gamma^n (\delta_{q+1}^{\sfrac{1}{2}} s, \xi)\\
\Psi (s, \xi) := & \partial_s \Gamma^n (\delta_{q+1}^{\sfrac{1}{2}} s, \xi)\, .\qedhere
\end{align*}
\end{proof}

\subsection{Estimates on $\|v - u_q\|_0$, $\|D (v-u_q)\|_0$ and $\|D^k v\|_0$} Taking into account Proposition \ref{p:Kuiper} we obviously have
\[
\|v - w\|_0 \leq C \delta_{q+1}^{\sfrac{1}{2}} \mu^{-1}\, ,
\]
whereas by \eqref{e:w_0}
\[
\|u_q-w\|_0 \leq C \delta_{q+1}^{\sfrac{1}{2}} \ell \leq C \delta_{q+1}^{\sfrac{1}{2}}\ell^{1-\sfrac{\alpha}{2}}\leq C \frac{\delta_{q+1}}{\delta_q^{\sfrac{1}{2}} \lambda_q}\, .
\]
We therefore conclude
\begin{equation}\label{e:v-u_0}
\|u_q - v\|_0 \leq C \delta_{q+1}^{\sfrac{1}{2}} \mu^{-1} + C \frac{\delta_{q+1}}{\delta_q^{\sfrac{1}{2}}\lambda_q}\, .
\end{equation}
By Lemma \ref{l:ugly_lemma_1} we easily see that
\begin{equation}\label{e:Dv-u_q}
\|D (u_q -v)\|_0 \leq C \delta_{q+1}^{\sfrac{1}{2}}
\end{equation}
and
\begin{equation}\label{e:D^kv_0}
\|D^k v\|_0 \leq C \delta_{q+1}^{\sfrac{1}{2}} \mu^{k-1} \qquad\qquad \mbox{for $k\in \{2,3\}$}\, .
\end{equation}
Observe also that, by \eqref{e:uniform_C1}, 
\begin{equation}\label{e:uniform_C1_bis}
\|Dv\|_0 \leq C\, .
\end{equation}

\subsection{Estimates on $\|E_1\|_0$ and $\|D E_1\|_0$}\label{s:E1} Observe first that due to Proposition \ref{p:Kuiper} (b) we have
\[
(Dw + \bA_1)^T (Dw + \bA_1) = w^\sharp e + \delta_{q+1} \rho^2 \nabla \Phi_1 \otimes \nabla \Phi_1\, .
\]
Using the notation $\sym P$ for the matrix $\frac{1}{2} (P + P^T)$ we can then write
\[
E_1 = 2\sym (Dw^T (\bB_1+\bC_1)) + 2\sym (\bA_1^T (\bB_1 + \bC_1)) + (\bB_1 + \bC_1)^T (\bB_1+\bC_1)\, .
\]
We notice that, from Lemma \ref{l:ugly_lemma_1} and the estimates \eqref{e:rho_DPhi_0} and \eqref{e:rho_DPhi_k}
on $\rho$ and $\Phi$, we conclude
\begin{align}
\|\bA_1\|_0 + \mu^{-1} \|D\bA_1\|_0 \leq C \delta_{q+1}^{\sfrac{1}{2}}\, ,\\
\|\bB_1\|_0 + \|\bC_1\|_0 + \mu^{-1} \big(\|D \bB_1\|_0 + \|D \bC_1\|_0\big) \leq C \frac{\delta_{q+1}^{\sfrac{1}{2}}}{\ell \mu}\, .
\end{align}
It is therefore obvious that, since $\ell \mu \geq 1$,
\begin{align}
\|E_1\|_0 &\leq \|Dw^T \bB_1\|+\|Dw^T \bC_1\|_0 + C \frac{\delta_{q+1}}{\ell \mu}\, ,\\
\|D E_1\|_0 & \leq \|D (Dw^T \bB_1)\|_0 + \|D(Dw^T\bC_1)\|_0 + C \delta_{q+1} \ell^{-1}\, .
\end{align}
We next compute
\[
Dw^T \bB_1 =  \frac{\delta_{q+1}^{\sfrac{1}{2}}}{\mu} B_1^t (Dw^T\, t_1) \otimes (\rho \nabla |\tau_1| + |\tau_1| \nabla \rho)\, .
\]
Therefore we conclude from Lemma \ref{l:ugly_lemma_1} that
\begin{align}
\|Dw^T \bB_1\|_0 \leq & C \frac{\delta_{q+1}}{\ell \mu}\, ,\\
\|D (Dw^T \bB_1)\|_0 \leq & C \delta_{q+1} \ell^{-1}\, .
\end{align}
Recalling moreover \eqref{e:n_1} we have
\[
D n_1 = \frac{D \nu_1}{|\tau_1|} - n_1 \otimes \frac{\nabla |\tau_1|}{|\tau_1|}\,
\]
and we also conclude that
\[
Dw^T \bC_1 = \frac{C_1^t}{\mu} Dw^T Dt_1 + \frac{C_1^n}{\mu} Dw^T \frac{D\nu_1}{|\tau_1|}\, .
\]
In particular
\[
\|Dw^T \bC_1\|_0 \leq \frac{C\delta_{q+1}}{\mu \ell} + C \frac{\delta_{q+1}^{\sfrac{1}{2}}}{\mu}\delta_{q+1}^{\sfrac{1}{2}}\ell^{-1} \leq C \frac{\delta_{q+1}}{\mu \ell}\, .
\]
Similarly we conclude
\[
\|D (Dw^T \bC_1)\|_0 \leq C \delta_{q+1} \ell^{-1}\, . 
\]
Thus we infer
\begin{align}
\|E_1\|_0 \leq & C \frac{\delta_{q+1}}{\ell \mu}\, ,\\
\|DE_1\|_0 \leq & C \delta_{q+1} \ell^{-1}\, .
\end{align}

\section{Estimates on $u_{q+1}$ and $E_2$}\label{s:est_2}
Our goal in this section is to estimate the $C^0$ norms of $u_{q+1}-v$, $D u_{q+1}$, $D^2 u_{q+1}$, $E_2$ and $D E_2$. We proceed in the same way as in the previous section and begin by defining the functions
\begin{align}
A_2^t & := \partial_\xi \Gamma^t \left( \delta_{q+1}^{\sfrac{1}{2}} |\tau_2| \rho, \lambda_{q+1} \Phi_2\right)\, ,\\
A_2^n& := \partial_\xi \Gamma^n \left( \delta_{q+1}^{\sfrac{1}{2}} |\tau_2| \rho, \lambda_{q+1} \Phi_2\right)\, ,\\
B_2^t& := \partial_s \Gamma^t \left( \delta_{q+1}^{\sfrac{1}{2}} |\tau_2| \rho, \lambda_{q+1} \Phi_2\right)\, ,\\
B_2^n& := \partial_s \Gamma^n \left( \delta_{q+1}^{\sfrac{1}{2}} |\tau_2| \rho, \lambda_{q+1} \Phi_2\right)\, ,\\
C_2^t& := \Gamma^t \left( \delta_{q+1}^{\sfrac{1}{2}} |\tau_2| \rho, \lambda_{q+1} \Phi_2\right)\, ,\\
C_2^n& := \Gamma^n \left( \delta_{q+1}^{\sfrac{1}{2}} |\tau_2| \rho, \lambda_{q+1} \Phi_2\right)\, 
\end{align}
and decomposing the derivative of $u_{q+1}$ as 
\begin{align}
&Du_{q+1} =  Dv + \underbrace{A_2^t\,  t_2 \otimes \nabla \Phi_2 + A_2^n\, n_2 \otimes \nabla \Phi_2}_{=:\bA_2}\nonumber\\
&+ \underbrace{\frac{\delta^{\sfrac{1}{2}}_{q+1}}{\lambda_{q+1}} (B^t_2\, t_2 + B^n_2\, n_2) \otimes (\rho \nabla |\tau_2| + |\tau_2| \nabla \rho)}_{=:\bB_2}
+ \underbrace{\frac{1}{\lambda_{q+1}} \left(C_2^t\, Dt_2 + C_2^n Dn_2\right)}_{=:\bC_2}\, .
\end{align}

\subsection{Second technical lemma} As before we collect the estimates of the $C^0$ norm of the derivatives of the
various quantities introduced above.

\begin{lemma}\label{l:ugly_lemma_2}
Assume $\tilde{C}$ is fixed so that Lemma \ref{l:convolution} holds and $\hat{C} >1$. If $a\geq a_0 (\alpha, g, b,c, \bar{C}, \hat{C})$ for a sufficiently large $a_0$, then there are constants $C$ (depending on $\alpha$ and $g$ but not on $\bar C$) such that
\begin{equation}
{C^{-1} \leq |\tau_2|\leq C}
\end{equation}
\begin{align}
\|D^k \nu_2\|_0&\leq C \delta_{q+1}^{\sfrac{1}{2}} \mu^k\qquad\qquad\;\; \mbox{for $k\in \{1,2\}$}
\end{align}
and, for $k\in \{0,1,2\}$,
\begin{align}
\|D^{k} t_2\|_0 + \|D^{k} \tau_2\|_0 + \|D^{k} n_2\|_0 &\leq C \ell^{-k} + C \delta_{q+1}^{\sfrac{1}{2}}\mu^k\\
\|D^k A_2^t\|_0 + \|D^k C_2^t\|_0 &\leq C \delta_{q+1} \lambda_{q+1}^k\\
\|D^k A_2^n\|_0 + \|D^k B_2^t\|_0 + \|D^k C_2^n\|_0 &\leq C \delta_{q+1}^{\sfrac{1}{2}} \lambda_{q+1}^k\\
\|D^k B_2^n\|_0 &\leq C \lambda_{q+1}^k\, .
\end{align}
\end{lemma}
\begin{proof}
The arguments are entirely similar to the ones of Lemma \ref{l:ugly_lemma_1}, where we only need to use the estimates \eqref{e:Dv-u_q} and \eqref{e:D^kv_0} on $D^k v$ proved in the previous section and the fact that $\lambda_{q+1} \geq \mu$.
\end{proof}

\subsection{Estimates on $\|u_{q+1}-v\|_0$, $\|D (u_{q+1}-v)\|_0$ and $\|D^2 u_{q+1}\|_0$}
The following estimates are straightforward consequences of Lemma \ref{l:ugly_lemma_2}:
\begin{align}
\|u_{q+1} - v\|_0 \leq & C \delta_{q+1}^{\sfrac{1}{2}} \lambda_{q+1}^{-1}\, ,\label{e:additional_1}\\
\|D u_{q+1} - Dv\|_0 \leq & C\delta_{q+1}^{\sfrac{1}{2}}\, ,\label{e:additional_2}\\
\|D^2 u_{q+1}\|_0 \leq & C \delta_{q+1}^{\sfrac{1}{2}} \lambda_{q+1}\, .\label{e:additional_3}
\end{align}

\subsection{Estimates on $\|E_2\|_0$ and $\|DE_2\|_0$}
Arguing as in Section \ref{s:E1} we easily see that
\begin{align}
\|E_2\|_0 \leq & C \delta_{q+1} \frac{\mu}{\lambda_{q+1}}\, ,\\
\|D E_2\|_0 \leq & C \delta_{q+1} \mu\, .
\end{align}

\section{Proof of Proposition \ref{p:main}, Conclusion}\label{s:iter_conclusion}

Recall that
\begin{equation}\label{e:def_mu_1}
\mu := \hat{C} \frac{\delta_{q+1} \lambda_{q+1}^\alpha}{\delta_{q+2} \ell}
\end{equation}
for an appropriately large constant $\hat{C}$, depending upon $\alpha$ and $g$ (in particular not on $a$). It then follows that
\[
\|E_1\|_0 + \lambda_{q+1}^{-1} \|D E_1\|_0 \leq \frac{\sigma_0}{12} \delta_{q+2} \lambda_{q+1}^{-\alpha}\, .
\]
Hence, (recall \eqref{e:gmoll1} and \eqref{e:gmoll2}) to achieve the estimates \eqref{e:me_0_q+1} and \eqref{e:me_1_q+1} we need to verify
\[
C \delta_{q+1} \frac{\mu}{\lambda_{q+1}} \leq \frac{\sigma_0}{12} \delta_{q+2} \lambda_{q+1}^{-\alpha}\, ,
\]
which however is implied by \eqref{e:key_parameters}, which is valid provided $a$ is chosen sufficiently large.
The three remaining inequalities \eqref{e:u_0_q+1}, \eqref{e:u_1_q+1} and \eqref{e:u_2_q+1} are implied by \eqref{e:v-u_0}-
\eqref{e:uniform_C1_bis} and \eqref{e:additional_1}-\eqref{e:additional_3}.

\section{Proof of Theorem \ref{t:main}}\label{s:main}

{

\subsection{Step 1}
By using the compactness of the domain $\overline{D}$ we may assume without loss of generality that $\bar{u}$ is uniformly strictly short, that is, $g-\bar{u}^\sharp e\geq 2\bar{\delta}$ in $\overline{D}$ for some $\bar{\delta}>0$. In a first step we will apply the classical Nash-Kuiper argument to obtain a good first approximation. 

To this end recall that there exist a finite number\footnote{Although the number $N$ in this decomposition depends on $\bar{\delta}>0$, there is a geometric constant $N_*$ such that for any $x\in\overline{D}$ at most $N_*$ of the functions $\phi_i$ are non-zero. Nevertheless, this information is not required for our purposes.}  of unit vectors $e_i\in\R^2$ and corresponding amplitudes $\phi_i\in C^{\infty}(\overline{D})$, $i=1,\dots,N$ such that 
$$
g-\bar{u}^\sharp e-\bar{\delta}e=\sum_{i=1}^N\phi_i^2e_i\otimes e_i\quad \textrm{ in }\overline{D}.
$$
Define iteratively the smooth mappings $\bar{u}_0:=\bar{u}$, $\bar{u}_1,\dots,\bar{u}_N=:\tilde u$ by setting, for $i=1,\dots,N$,
\[
\tau_i := D\bar{u}_{i-1} (D\bar{u}_{i-1}^TD\bar{u}_{i-1})^{-1} e_i\, ,\quad 
\nu_i := \frac{\partial_{x_1} \bar{u}_{i-1} \times \partial_{x_2} \bar{u}_{i-1}}{|\partial_{x_1} \bar{u}_{i-1} \times \partial_{x_2} \bar{u}_{i-1}|}\, ,
\]
\begin{equation*}
t_i := \frac{\tau_i}{|\tau_i|^2}\,,\quad n_i := \frac{\nu_i}{|\tau_i|}\, .
\end{equation*}
and
\begin{equation}\label{e:Nashtwist}
\bar{u}_i (x) := \bar{u}_{i-1}(x)  + \frac{1}{\mu_i} \Gamma^t \big(\varphi_i|\tau_i|, \mu_ie_i\cdot x\big) t_i + \frac{1}{\mu_i} \Gamma^n \big(\varphi_i|\tau_i|, \mu_ie_i\cdot x\big) n_i\, . 
\end{equation}
Here the frequencies $1\leq \mu_1\leq \mu_2\leq\dots\leq \mu_N$ will be inductively defined as follows. 
Let
$$
E_i=\bar{u}_i^\sharp e-\bar{u}_{i-1}^\sharp e-\phi_i^2e_i\otimes e_i
$$
so that $\bar{u}_N^\sharp e=g-\bar{\delta}e+\sum_{i=1}^NE_i$. As in Section \ref{s:est_1} we can estimate $E_1$ as
$$
\|E_1\|_0\leq \frac{C(\bar{u})}{\mu_1}\,,\quad \|E_1\|_1\leq C(\bar{u}),
$$
where $C(\bar{u})$ is a constant depending on $\bar{u}$. By interpolation we also have 
$$
\|E_1\|_{\alpha}\leq \frac{C(\bar{u})}{\mu_1^{1-\alpha}},
$$
and moreover $\|\bar{u}-\bar{u}_1\|_0\leq C\mu_1^{-1}$. 
Therefore we can choose $\mu_1$ so that 
$$
\|E_1\|_{\alpha}\leq \frac{\sigma_1}{2N} \bar \delta,\quad \|\bar{u}-\bar{u}_1\|_0\leq \frac{\varepsilon}{2N}.
$$
Continuing, analogously we obtain
$$
\|E_2\|_0\leq \frac{C(\bar{u},\mu_1)}{\mu_2}\,,\quad \|E_2\|_1\leq C(\bar{u},\mu_1),
$$
and hence choose $\mu_2$ so that 
$$
\|E_2\|_{\alpha}\leq \frac{\sigma_1}{2N} \bar \delta,\quad \|\bar{u}_2-\bar{u}_1\|_0\leq \frac{\varepsilon}{2N}.
$$
In a similar manner we can inductively choose $\mu_i$, $i=3,\dots,N$ so that eventually we obtain
$$
\|g-\bar{\delta}e-\tilde{u}^\sharp e\|_\alpha\leq \sum_{i=1}^N\|E_i\|_{\alpha}\leq \frac{\sigma_1}{2} \bar \delta
$$
and
$$
\|\bar{u}-\tilde{u}\|_0\leq \frac{\varepsilon}{2}.
$$

\begin{remark}\label{r:embedding}
The construction above can be easily adapted to the case when $\bar{u}$ is an embedding, and in this case also $\tilde u$ will be an embedding. This is of course well-known and has been proved by Nash and Kuiper. In order to keep our paper self-contained, we nevertheless include here a short proof.

Since the construction of $\tilde u$ from $\bar{u}$ involves finite number of steps, it suffices to ensure that at each step $\bar{u}_{i}$ remains an embedding, i.e. no self-intersections are introduced. To show this, we proceed by induction and assume that $\bar{u}_{i-1}$ is an embedding. By using Proposition \ref{p:Kuiper} and the choice of vectors $t_i,n_i$ we can write \eqref{e:Nashtwist} as
$$
\bar{u}_i(x):=\bar{u}_{i-1}(x)+\frac{1}{\mu_i}w_i(x,\mu_i x),
$$
where $w_i=w_i(x,\xi)$ satisfies
\begin{equation*}
\begin{split}
\bigl[D\bar{u}_{i-1}(x)+&\partial_{\xi}w_i(x,\mu_i z)\bigr]^T\bigl[D\bar{u}_{i-1}(x)+\partial_{\xi}w_i(x,\mu_i z)\bigr]=\\
&=D\bar{u}_{i-1}(x)^TD\bar{u}_{i-1}(x)+\phi_i^2(x)e_i\otimes e_i.
\end{split}
\end{equation*}
for any $x,z$. In particular, since $\bar{u}_{i-1}$ is an immersion, there exists $\omega_1>0$ 
so that 
\begin{equation}\label{e:localinjectivity}
\bigl|(D\bar{u}_{i-1}(x)+\partial_{\xi}w_i(x,\mu_i z)\bigr)e|\geq \bigl| D\bar{u}_{i-1}(x)e\bigr|\geq \omega_1|e|
\end{equation}
for any vector $e$.  

Next, let $x,y\in \overline{D}$. By Taylor's theorem and the mean value theorem there exists $z$ on the line segment $[x,y]$ such that
$$
\bar{u}_i(x)-\bar{u}_i(y)=D\bar{u}_{i-1}(x)(x-y)+\partial_{\xi}w_i(x,\mu_i z)(x-y)+\tilde E,
$$
where 
$$
|\tilde E|\leq C\left(|x-y|^2+\frac{1}{\mu_i}|x-y|\right),
$$
and $C$ is a constant depending on the functions $\bar{u}_{i-1}(x)$ and $w_i(x,\xi)$ but not on $\mu_i$. 
Let $\rho=\frac{\omega_1}{4C}$ and choose $\mu_i>\rho^{-1}$. From \eqref{e:localinjectivity} we deduce that if $|x-y|\leq \rho$, then 
$$
|\bar{u}_i(x)-\bar{u}_i(y)|\geq \frac{\omega_1}{2}|x-y|.
$$
On the other hand, since $\bar{u}_{i-1}$ is assumed to be globally injective and $\overline{D}$ is compact, there exists $\omega_2>0$ such that
$$
|\bar{u}_{i-1}(x)-\bar{u}_{i-1}(y)|\geq \omega_2|x-y|\quad\textrm{ for all }|x-y|\geq\rho.
$$
Since obviously $\|\bar{u}_i-\bar{u}_{i-1}\|_0\leq C\mu_i^{-1}$, it follows that for sufficiently large $\mu_i$ we will also have
$$
|\bar{u}_{i}(x)-\bar{u}_{i}(y)|\geq \omega_2|x-y|\quad\textrm{ for all }|x-y|\geq\rho.
$$
In summary, we have shown that, by choosing $\mu_i$ sufficiently large, we can ensure that $\bar{u}_i$ is also an embedding. 
\end{remark}

\subsection{Step 2} In step 1 we obtained a good approximation $\tilde{u}$ in the sense that 
\eqref{e:me_0_alpha_q} from Proposition \ref{p:main} is satisfied. However, although $\tilde{u}$ is smooth, we have no information on the size of the second derivatives $D^2\tilde{u}$. Therefore in this step we obtain a further approximation $u_0$, where in addition second derivatives are controlled so that this second approximation can then be used as the starting point of an iteration with Proposition \ref{p:main}.

}

In this step we assume in addition\footnote{Indeed it could be checked directly that \eqref{e:abc_2} implies \eqref{e:abc_3} and hence \eqref{e:abc_3} is superfluous: however, proceding as we do we can spare the reader a slightly tedious computation}
\begin{equation}\label{e:abc_3}
c > \frac{2}{1-2\alpha} + \frac{1}{2b}.
\end{equation}
We show that, no matter how large $a$ is chosen, there is a map $u_0$ satisfying the assumptions \eqref{e:me_0_alpha_q} and
\eqref{e:Du_2_q} of Proposition \ref{p:main}, where the constant $\bar{C}$ in the latter estimate is however {\em independent} of $a$ (because it depends only on $g$ and $\tilde u$). We proceed as in Section \ref{s:iter_start}, except no regularization step is necessary this time. 
We set
\[
h := \frac{g- \tilde{u}^\sharp e}{\bar{\delta}} - \frac{\delta_1}{\bar{\delta}}e
\]
and apply Proposition \ref{p:conformal} to find ($C^3$) $\Phi_1$, $\Phi_2$ and $\rho$ so that
\[
h := \rho^2 (\nabla \Phi_1 \otimes \nabla \Phi_1 + \nabla \Phi_2 \otimes \nabla \Phi_2)\, .
\]
We then define 
\begin{equation*}
\tau_1 := D\tilde{u} (D\tilde{u}^T D\tilde{u})^{-1} \nabla \Phi_1\, ,
\end{equation*}
\begin{equation*}
\nu_1 := \frac{\partial_{x_1} \tilde{u} \times \partial_{x_2} \tilde{u}}{|\partial_{x_1} \tilde{u} \times \partial_{x_2} \tilde{u}|}\, ,
\end{equation*}
and
\begin{align*}
 t_1 := \frac{\tau_1}{|\tau_1|^2}\,,\quad n_1 := \frac{\nu_1}{|\tau_1|}\, .
\end{align*}
Hence we set
\begin{equation}\label{e:vbis}
v = \tilde{u} + \frac{1}{\mu} \Gamma^t \left( \bar\delta^{\sfrac{1}{2}} |\tau_1| \rho, \mu \Phi_1\right) t_1 + \frac{1}{\mu} \Gamma^n 
\left(\bar \delta^{\sfrac{1}{2}} |\tau_1|\rho, \mu \Phi_1\right)\, n_1\, .
\end{equation}
Then we define
\begin{equation*}
\tau_2 := Dv (Dv^T Dv)^{-1} \nabla \Phi_2\,,
\end{equation*}
\begin{equation*}
\nu_2 := \frac{\partial_{x_1} v \times \partial_{x_2} v}{|\partial_{x_1} v\times \partial_{x_2} v|}\, ,
\end{equation*}
and
\begin{align*}
t_2:= \frac{\tau_2}{|\tau_2|^2}\, ,\quad n_2 := \frac{\nu_2}{|\tau_2|}\, .
\end{align*}
The map $u_0$ is finally given by
\begin{equation}\label{e:u_0}
u_0= v + \frac{1}{\lambda} \Gamma^t \left( \bar \delta^{\sfrac{1}{2}} |\tau_2| \rho, \lambda\Phi_2\right) t_2 + \frac{1}{\lambda} \Gamma^n \left(\bar \delta^{\sfrac{1}{2}} |\tau_2|\rho, \lambda \Phi_2\right) n_2\, .
\end{equation}
Again we assume $\lambda \geq \mu \geq 1$. Analogous computations to the ones in Sections \ref{s:est_1} and \ref{s:est_2} lead 
to the estimates
\[
\|g-(u_{0}^\sharp e + \delta_1 e)\|_\alpha \leq C \bar \delta^{\sfrac{1}{2}} \mu^{2\alpha-1} + C \bar \delta \mu \lambda^{\alpha-1}\, 
\]
\[
\|D^2 u_0\|_0\leq C \bar \delta^{\sfrac{1}{2}} \lambda\, ,
\]
where the constant $C$ depends only on $\tilde{u}$ and $g$. We thus set
\[
\mu:= C_1 \delta_1^{-1/(1-2\alpha)} \quad \mbox{and}\quad \lambda := C_2 \mu^{1/(1-\alpha)} \delta_1^{-1/(1-\alpha)}\, .
\]
For a sufficiently large choice of $C_2$ and $C_1$ we then achieve \eqref{e:me_0_alpha_q} (recall that $\bar\delta < 1$. 

Clearly
\[
\|D^2u_0\|_0 \leq C_3 \delta_1^{- 2/(1-2\alpha)}\, ,
\]
for a constant $C_3$ which depends only upon $\tilde{u}, g$ and $\alpha$. In order to show that \eqref{e:Du_2_q} is satisfied with a constant
$\bar{C}$ independent of $a$, it suffices to show that
\[
\delta_1^{- 2/(1-2\alpha)} \leq \delta_0^{\sfrac{1}{2}} \lambda_0\, .
\]
Taking the logarithms in base $a$ the latter inequality is implied by
\[
cb \geq \frac{1}{2} + \frac{2}{1-2\alpha} b\, .
\]

\subsection{Step 3} {Finally we are ready for the iteration based on Proposition \ref{p:main}.} Fix any $\alpha$, $b$ and $c$ which satisfies \eqref{e:abc_1}, \eqref{e:abc_2} and \eqref{e:abc_3}. Then, for any sufficiently large $a$, we can construct a map $u_0$ as in the previous step which satisfies $\|\bar{u} - u_0\|_0< \frac{\varepsilon}{2}$ and the assumptions of Proposition \ref{p:main}, with a constant $\bar{C}$ which does not depend on $a$. We can apply Proposition \ref{p:main} to generate $u_1$. Using \eqref{e:interpolation_explicit} we conclude
\begin{align}
\|g_1 - u_1^\sharp e\|_\alpha \leq & \|g_1-u_1^\sharp e\|_0 + 2\|g_1-u_1^\sharp e\|_0^{1-\alpha} \|D (g_1-u_1^\sharp e)\|_0^\alpha
\nonumber\\
\leq & \sigma_0 \delta_2\, .
\end{align}
Hence $u_1$ satisfies again the assumptions of Proposition \ref{p:main}. More generally, the proposition can be applied inductively to generate a sequence $(u_q)_{q\geq0}$. Observe that \eqref{e:u_0_q+1}-\eqref{e:u_2_q+1} imply that
\begin{itemize}
\item $(u_q)_{q\geq0}$ converges uniformly to a map $u$ which (assuming $a$ sufficiently large) satisfies $\|u_0 - u\|_0 < \frac{\varepsilon}{2}$. By assumption on $u_0$ we therefore have $\|\bar u - u \|_0 < \varepsilon$.
\item Interpolating $\|D (u_{q+1} - u_q)\|\leq C \delta_{q+1}^{\sfrac{1}{2}}$ and 
\begin{align*}
\|D^2 (u_{q+1} - u_q)\|_0 \leq &\|D^2 u_{q+1}\|_0 + \|D^2 u_q\|_0 \leq \bar{C} \delta_{q+1}^{\sfrac{1}{2}} \lambda_{q+1} +
\bar{C} \delta_q^{\sfrac{1}{2}} \lambda_q\\
\leq&  2 \bar{C} \delta_{q+1}^{\sfrac{1}{2}} \lambda_{q+1}
\end{align*}
shows
\[
\|D (u_{q+1} - u_q)\|_\beta \leq C^\star \delta_{q+1}^{\sfrac{1}{2}} \lambda_{q+1}^\beta\, ,
\]
for a constant $C^\star$ which depends on $\alpha, g$ and $\bar{C}$.
Hence using the definitions \eqref{e:delta_lambda} of $\delta_q$ and $\lambda_q$ we can see that if $\beta < \frac{1}{2bc}$
then $(u_q)_{q\geq 0}$ is a Cauchy sequence on $C^{1,\beta}$.
\end{itemize}
We next show that, if $\alpha$ is chosen arbitrarily small, $bc$ can be chosen arbitrarily close to $\frac{5}{2}$, which in turn implies that $\beta$ can be made arbitrarily close to $\frac{1}{5}$. Indeed if we let $\alpha\downarrow 0$, the conditions \eqref{e:abc_1}, \eqref{e:abc_2} and \eqref{e:abc_3} become, respectively
\begin{align}
b > & 1\label{e:abc_11}\\
c > & \frac{4b^2-3b-1}{2b (b-1)} = 2 + \frac{1}{2b}\label{e:abc_12}\\
c > & 2 + \frac{1}{2b}\, .
\end{align}
This completes the proof.

\section{Proof of Theorem \ref{t:main2}}\label{s:main2}

First of all we notice that, by classical extension theorems, the first statement can be reduced to Theorem \ref{t:main}: it suffices
to extend both $g$ and $\bar{u}$ smoothly from $\bar D_1$ to $\bar D_2$. The extended map is not necessarily short for the extended metric, but we
can ensure this if we add to the extension of $g$ a tensor of the form $\varphi (|x|) e$, where $\varphi$ is a rapidly growing $C^\infty$ function which vanishes identically on $[0,1]$. 

Next, observe that the arguments of the Steps 2,3 and 4 in Section \ref{s:main}, combined with the extension trick outlined above give in fact the following corollary. 

\begin{corollary}\label{c:main}
Let $g$ be a $C^2$ metric on $\overline{D}_1$. Then there are positive constants $C_0, \bar{c}$ and $\bar{\eta}$ with the following properties. Assume that 
\begin{itemize}
\item[(i)] $\underline u: \overline{D}_1 \to \R^3$ is $C^\infty$,
\item[(ii)] $\|g - (\underline u^\sharp e + 2 \eta e)\|_0\leq \bar{c} \eta$ for some $\eta\in (0, \bar\eta)$.
\end{itemize}
Then {for any $\varepsilon >0$ and $\delta>0$} there is an {isometric map} $u\in C^{1, 1/5-\delta} (\overline{D}_1)$ such that
$\|Du - D\underline u\|_0 \leq C_0 \eta^{\sfrac{1}{2}}$ and $\|u- \underline u\|_0 \leq \varepsilon$.
\end{corollary}

With this corollary at hand we can prove Theorem \ref{t:main2} in two easy steps. {In the proof we will restrict to the case of embeddings, the case of immersions can be obtained by easy modifications.}

\medskip

{\bf Proof of Theorem \ref{t:main2} for embeddings.} 
Let $g$ be a $C^2$ metric on $\overline{D}_1$ and $\bar{u}\in C^1(\overline{D}_1,\R^3)$ a short embedding. By a simple rescaling and mollification we may assume without loss of generality that $\bar{u}$ is smooth and strictly short. 
Next, fix $\omega>0$ such that $g \geq 16 \omega^2 e$ and choose $\eta>0$ such that $\eta \leq \min \{\omega^2, \bar \eta\}$ and $C_0 \eta^{\sfrac{1}{2}} \leq \omega$.  

As in Step 1 of the proof of Theorem \ref{t:main} (including Remark \ref{r:embedding}) we first construct a smooth embedding $\underline{u}$ with
$$
\|\underline{u} - \overline u\|_0 < \frac{\varepsilon}{2}
$$
and such that
$$
\|g - (\underline u^\sharp e + 2 \eta e)\|_0\leq \bar{c} \eta.
$$
Then the assumptions of Corollary \ref{c:main} are satisfied and we obtain 
$u\in C^{1, 1/5-\delta} (\overline{D}_1)$ with $u^\sharp e = g$ and such that
$\|Du - D\underline u\|_0 \leq C_0 \eta^{\sfrac{1}{2}}$ and $\|u- \underline u\|_0 \leq \varepsilon/2$.

To complete the proof, it remains to show that the map $u$ is an embedding. We again remark that this argument is well-known and is contained in the works of Nash and Kuiper.
First of all, since $\underline{u}$ is $C^1$, there exists $\rho>0$ such that $|D\underline u (z)- D\underline u (y)|\leq \omega$ if $|z-y|\leq \rho$. On the other hand, since $\underline {u}$ is an embedding, then there is $\zeta>0$ such that $|\underline u(z)- \underline  u(y)|\geq 3\zeta$ if $|z-y|\geq \rho$. 

To show global injectivity, we now observe that
\[
|u(z)- u(y)|\geq |\underline u (z) - \underline u (y)| - 2\varepsilon \geq 3\zeta - 2\zeta = \zeta \qquad \mbox{when $|z-y|\geq \rho$.}
\]
On the other hand, if $|z-y| \leq \rho$ we know that
\[
|Du (z) - Du (y)| \leq |D\underline u (z) - D \underline u (y)| + 2 \omega \leq 3 \omega\, ,
\]
and hence, using Taylor's formula
\[
|u(z) - u(y) - Du (z)(z-y)| \leq 3 \omega |z-y|\, .
\]
We therefore can estimate
\[
|u(z)-u(y)|\geq |Du (z)(z-y)| - 3 \omega |z-y|
\]
But $u^\sharp e = g \geq 16 \omega^2 e$ implies $|Du (z)(z-y)|^2 \geq 16 \omega^2|z-y|^2$, which in turn shows $|u(z)-u(y)|\geq \omega |z-y|> 0$. 

This completes the proof of Theorem \ref{t:main2}.

\appendix

\section{Proof of Proposition \ref{p:conformal}}

\subsection{Beurling and Cauchy transforms} We will need the following two classical integral operators to construct the coordinate transformation of Proposition \ref{p:conformal}. In this section we use the standard notation $z=x+iy$ for complex numbers.
Moreover, we recall two standard differential operators $\partial_z = \frac{1}{2}\left( \partial_x-i\partial_y\right)$ and $\partial_{\bar{z}} = \frac{1}{2}\left( \partial_x+i\partial_y\right)$.
 \begin{definition}
 Suppose $G\subset \C$ is a bounded smooth open set and $f:G\rightarrow \C$ a function. For $z_0\in \C$  we define the \emph{Cauchy transform}
 \[ \sC_G[f](z_0) \coloneqq -\frac{1}{\pi} \int_G \frac{f(z)}{z-z_0}\, dx\, dy \]
 and the \emph{Beurling transform}
 \[ \sS_G[f](z_0) \coloneqq -\frac{1}{\pi} \int_G \frac{f(z)}{(z-z_0)^2}\, dx\, dy\, .\]
 The latter integral must be understood as a Cauchy principal value, in case it exists. As it is easy to check, the H\"older continuity of $f$ is enough to guarantee its existence at every point.
 \end{definition}
\begin{remark}
In the literature the terms Cauchy and Beurling transforms are often used only the operators $\sC_{\C}$ and $\sS_{\C}$.
\end{remark}
In the book of I. N. Vekua \cite{vek62} one can find the following very important properties of the operators $\sC_G$ and $\sS_G$ (cf.  \cite[Theorem 1.32]{vek62}).
 \begin{lemma}\label{vekua}	
 Let $N\in \N$, $0<\alpha<1$, $G\subset \C$ bounded and $f\in C^{N,\alpha}(\overline{G})$. Then we have
 \begin{enumerate}[(i)]
 \item $\sC_G[f] \in C^{N+1,\alpha}(\overline{G})$ and $\sS_G[f] \in C^{N,\alpha}(\overline{G})$;\\
 \item $\frac{\partial}{\partial\overline{z}}\sC_G[f](z) = f(z)$ and $\frac{\partial}{\partial z}\sC_G[f](z) = \sS_G[f](z)$ $\forall z\in G$;\\
 \item There exists a constant $C_{N,\alpha}$ such that
 \[\|\sS_G[f]\|_{N + \alpha}\leq \|\sC_G[f]\|_{N+1 + \alpha}\leq C_{N,\alpha}\|f\|_{N + \alpha}\, .\]
 \end{enumerate}
 \end{lemma}

Property (iii) will be key in order to prove Proposition \ref{p:conformal}. Observe that we can easily find solutions of equations of the type $f_{\overline{z}} =g$ by setting $f=\sC_G[g]$. 
Moreover, we have $\partial_{\overline{z}} \sS_G[f]=f_z$, so  $\sS_G$ links the two operators $\partial_{\overline{z}}$ and $\partial_z$.
To prove regularity and get good estimates we need one more thing, namely that under suitable circumstances the transforms commute with differentiation. This will be the content of Corollary \ref{commute}. 
 \begin{lemma}\label{c-p-formula}
 Let $r>0$ and $f\in C^1(\bar D_r)$. Then for any $z_0\in D_r$ we have the identities
    \begin{align}\label{CPF}
	f(z_0) &= \dfrac{1}{2\pi i} \int_{\partial D_r} \dfrac{f(z)}{z-z_0}\, dz - \frac{1}{\pi}\int_{D_r}\frac{f_{\overline{z}}(z)}{z-z_0}\, dx\, dy\, ,\\ 
	\frac{1}{\pi}\int_{D_r}\frac{f(z)}{(z-z_0)^2}\, dx\, dy &=\frac{1}{\pi}\int_{D_r}\frac{f_z(z)}{z-z_0}\, dx\, dy+\frac{1}{2\pi i}\int_{\partial D_r}\frac{f(z)}{z-z_0}\, d\overline{z}\, .
    \end{align}
 \end{lemma}
 \begin{proof}
 Take a fixed $z_0\in D_r$ and look at the differential one-form $ \omega = \frac{dz}{z-z_0}$. We can see that
      \[d(\omega f) = \frac{f_{\overline{z}}}{z-z_0}\, d\overline{z}\wedge dz = 2i \frac{f_{\overline{z}}}{z-z_0}\, dx\wedge dy\, ,\]
 hence by Stoke's theorem we have 
	\begin{equation}\label{1}
	    2i\int_{D_r\setminus D_\varepsilon}\frac{f_{\overline{z}}(z)}{z-z_0}\, dx\, dy = \int_{\partial D_r} \frac{f(z)}{z-z_0}\, dz - \int_{\partial D_\varepsilon} \frac{f(z)}{z-z_0}\, dz\, .
	\end{equation}
 We can easily compute
     \[\lim\limits_{\varepsilon\rightarrow 0}{\int_{\partial D_\varepsilon} \frac{f(z)}{z-z_0}dz} = 2\pi i f(z_0)\, ,\] 
 and therefore passing to the limit $\varepsilon\rightarrow 0 $ in \eqref{1} yields the first statement; 
 the same reasoning applied to the one-form $\tilde{\omega} = \frac{d\overline{z}}{z-z_0}$ shows the second. 
 \end{proof}

\begin{remark}
 Observe that if we define $\Psi(z_0) = \frac{1}{2\pi i} \int_{\partial D_r} \frac{f(z)}{z-z_0}\, dz$ then the statements of the previous lemma can be rewritten as
      \begin{align*}
	  f(z_0) &= \Psi(z_0) + \sC_{D_r}[f_{\overline{z}}](z_0)\, ,\\
	  \sS_{D_r}[f](z_0) &= \sC_{D_r}[f_z](z_0)-\frac{1}{2\pi i}\int_{\partial D_r}\frac{f(z)}{z-z_0}\, d\overline{z}\, .
      \end{align*}
 \end{remark}

 \begin{remark}
It follows from Lemma \ref{c-p-formula} that if $f\in C^1_0(\bar D_r)$, then 
    \begin{enumerate}[(i)]
    \item  $\sC_{D_r}[f_{\overline{z}}] = f\, ,$
    \item  $\sC_{D_r}[f_z] = \sS_{D_r}[f]\, .$
    \end{enumerate}
Combining these two identities with Lemma \ref{vekua} we can derive 
 \[
(\sC_{D_r}[f])_z = \sS_{D_r}[f] = \sC_{D_r}[f_z]\, ,
\]
\[
(\sC_{D_r}[f])_{\overline{z}}  = f = \sC_{D_r}[f_{\overline{z}}]\, ,
\]
and 
\[
(\sS_{D_r}[f])_z =  (\sC_{D_r}[f_z])_z =  \sS_{D_r}[f_z]\, ,
\]
\[
(\sS_{D_r}[f])_{\overline{z}} = (\sC_{D_r}[f_z])_{\overline{z}} = \sC_{D_r}[(f_z)_{\overline{z}}]   = \sC_{D_r}[(f_{\overline{z}})_z] = \sS_{D_r}[f_{\overline{z}}]\, .
\]
This shows that for (sufficiently regular) functions with compact support in $D_r$, the operators $\sC_{D_r}$ and $\sS_{D_r}$ commute with any linear differential operator $\sD$ with constant coefficients. The regularity needed on the function is only linked to the order of the operator $\sD$.
\end{remark}

We summarize the latter discussion in the following
 
\begin{corollary}\label{commute}
Let $r>0$ and let $\sD$ be a linear differential operator with constant coefficients of order $k$. Then we have the following identities on $C^k_c (D_r)$:
    \begin{enumerate}[(i)]
    \item  $ \sD\circ \sC_{D_r} = \sC_{D_r}\circ \sD$ and $\sD\circ \sS_{D_r} = \sS_{D_r}\circ \sD$;
    \item  $ \partial_{\zb} \circ \sC_{D_r} = \sC_{D_r} \circ \partial_{\zb} = Id$ and $\partial_{z} \circ \sC_{D_r} = \sC_{D_r} \circ \partial_{z} = \sS_{D_r}$;
    \item  $ \partial_{\zb} \circ \sS_{D_r} = \sS_{D_r} \circ \partial_ {\zb} = \partial_{z}$.
    \end{enumerate}
\end{corollary}

\subsection{Beltrami's equation} With the various properties above established, we take a fundamental step to the proof of Proposition \ref{p:conformal}. As usual we denote by $C^{N, \alpha}_0 (\overline{D}_r)$ the closure of $C^{N,\alpha}_c\left(D_r\right)$ in the H\"older space $C^{N,\alpha} (\overline{D}_r)$.

\begin{lemma}\label{linhombel}
 Let $r\geq1$, $N \in \N, N\geq 1$, $0<\beta\leq\alpha<1$, $\mu,h \in C^{N,\alpha}_0 (\bar D_r)$. Then there exist constants $C(N,r,\alpha,\beta)$, $c(N,r,\alpha,\beta)$ and $\bar{C}(\alpha)$ such that if $\|\mu\|_\alpha \leq c$ there exists a solution $\Phi\in C^{N+1,\alpha}(\bar D_r)$ to 
 \begin{equation}\label{inhombel}
  \Phi_{\overline{z}} - \mu\Phi_z = h
 \end{equation}
with 
  \begin{align}
    \|\Phi\|_{1 +\alpha} &\leq \bar{C}\|h\|_{\alpha}\, ,\\
    \|D^k\Phi\|_{1+\beta}& \leq C \left(\|D^{k}h\|_\beta+\|D^{k}\mu\|_\beta\|h\|_\beta\right)\, ,
  \end{align}
  for any $1\leq k \leq N$. 
\end{lemma}
\begin{proof} By a standard approximation argument, it suffices to prove the lemma under the assumption that the supports of $\mu$ and $h$ are compactly contained in $D_r$.

In order to simplify our notation we will use $\sS$ and $\sC$ in place of $\sS_{D_r}$ and $\sC_{D_r}$. We know (thanks to Lemma \ref{vekua}) that $\sS:C^{0,\alpha}(\bar D_r)\rightarrow C^{0,\alpha}(\bar D_r)$ as well as $\sC:C^{0,\alpha}(\bar D_r)\rightarrow C^{1,\alpha}(\bar D_r)$ 
and that there exist two constants $C_\alpha, C_\beta$ (wlog $C_\alpha,C_\beta > 1$) such that
  \begin{align*}
    \|\sS[f]\|_{\alpha} &\leq \|\sC[f]\|_{1+\alpha}\leq C_\alpha \|f\|_\alpha\, ,\\
    \|\sS[f]\|_\beta &\leq \|\sC[f]\|_{1+\beta}\leq C_\beta\|f\|_\beta\, .
  \end{align*}
Consider the operator 
    \[\sL_\alpha :C^{0,\alpha}(\overline{D_r})\rightarrow C^{0,\alpha}(\overline{D_r}), \quad f \mapsto h+\mu \sS[f]\]
We have
    \[ \|\sL_\alpha (f_1)-\sL_\alpha (f_2)
\|_{\alpha} \leq \|\mu\|_\alpha C_\alpha\|f_1-f_2\|_\alpha\, . \]
So, if
\[ 
\|\mu\|_{\alpha} \leq \frac{1}{2C_\alpha}
\]
then $\sL_\alpha$ has a unique fixpoint $f\in C^{0,\alpha}(\bar D_r)$. This means 
\[ 
f=h+\mu \sS[f]\, ,
\]
from which we deduce 
    \[ \|f\|_{\alpha} \leq \frac{\|h\|_{\alpha}}{1-\|\mu\|_{\alpha}C_\alpha} \leq 2\|h\|_{\alpha}\]
and 
  \begin{equation*}
    f = \left( Id-\mu \sS \right)^{-1}h = \sum\limits_{n\geq0}\left( \mu \sS \right)^{n}h \eqqcolon \sum\limits_{n\geq0} \omega_n\, .
  \end{equation*}
This shows in particular that $f$ is compactly supported. Using Corollary \ref{commute} one can show by induction that for any $1\leq k\leq N$ and any $n\geq1$
  \begin{equation}\label{conv}
    \|D^k\omega_n\|_{\alpha} \leq \tilde{C}C_\alpha(2\tilde{C}C_\alpha\|\mu\|_{\alpha})^{n-1}\left( \|\mu\|_{\alpha}\|D^{k}h\|_\alpha+\|D^{k}\mu\|_\alpha\|h\|_\alpha \right)\, ,
   \end{equation}
where $\tilde{C}$ is the constant in \eqref{e:product}. Therefore, if we require 
  \begin{equation}\label{mualpha}
    \|\mu\|_{\alpha} \leq \left( 4\tilde{C}C_\alpha C_\beta (2r)^{\alpha-\beta}\right)^{-1}\, ,
   \end{equation}
then the series 
      \[ \sum\limits_{n\geq0} D^k \omega_n \]
converges uniformly in $C^{0,\alpha}(\overline{D_r})$ to $D^kf$, hence $f\in C^{N,\alpha}(\overline{D_r})$. Moreover, by the same argument
 \begin{align}
     \|D^kf\|_{\beta} &\leq \tilde{C}C_\beta ( \|\mu\|_{\beta}\|D^{k}h\|_\beta+\|D^{k}\mu\|_\beta\|h\|_\beta)\sum\limits_{n\geq1}\left( 2\tilde{C}C_\beta(2r)^{\alpha-\beta}\|\mu\|_{\alpha}\right)^{n-1}\nonumber\\
& \qquad\qquad +\|D^{k}h\|_\beta
		      \leq C\left( \|D^{k}h\|_\beta+\|D^{k}\mu\|_\beta\|h\|_\beta \right)\, ,\label{DMf}
 \end{align}
with the help of \eqref{mualpha}, where the constant  $C$ depends only on $N$, $r$, $\alpha$ and $\beta$. Now we define 
     \[ \Phi(z) = \sC[f](z),\quad z\in D_r\,. \]
By property (iii) of Lemma \ref{vekua} we have
     \[ \Phi_{\zb} = f, \Phi_{z} = \sS[f]\, ,\]
hence 
     \[ \Phi_{\zb} - \mu \Phi_{z} = f- \mu \sS[f] = (Id-\mu \sS)f = h \, ,\]
so the function $\Phi$ solves \eqref{inhombel} and satisfies  
     \[\|\Phi\|_{1+\alpha} \leq C_\alpha \|f\|_{\alpha}\leq 2C_{\alpha}\|h\|_{\alpha}\, .\]
Since $D^k\Phi = \sC[D^kf]$ by Corollary \ref{commute} we get by recalling \eqref{DMf}
     \[ \|D^k\Phi\|_{1+\beta} \leq C_\beta\|D^kf\|_{\beta} \leq C \left(  \|D^{k}h\|_\beta+\|D^{k}\mu\|_\beta\|h\|_\beta\right)\label{D^kf}\, .\]
This shows the claim.
\end{proof}

We immediately get the following 
\begin{corollary}\label{Beltrami}
   Let $r\geq1$, $N \in \N, N\geq 1$, $0<\beta\leq\alpha<1$, $\mu \in C^{N,\alpha}_0 (\bar D_r)$. Then there exist constants $C(N,r,\alpha,\beta)$, $c(N,r,\alpha,\beta)$ and $\bar{C}(\alpha)$
   such that if $\|\mu\|_{\alpha} \leq c$ there exists a solution $\Phi\in C^{N+1,\alpha}(\bar D_r)$ to the Beltrami equation
      \begin{equation}\label{Beltramieq}
	\Phi_{\overline{z}} = \mu\Phi_z
      \end{equation}     
with 
      \begin{align}
	\|\Phi(z)-z\|_{1+\alpha} & \leq \bar{C}\|\mu\|_{\alpha}\, ,\\
	\|D^k\left(\Phi(z)-z\right)\|_{1+\beta} &\leq C\|D^{k}\mu\|_{\beta}\, ,
      \end{align}
for any $1\leq k\leq N$. 
\end{corollary}
\begin{proof}
 In the Lemma \ref{linhombel} choose $h=\mu$ to recover a constant $c$ such that if $\|\mu\|_{\alpha}\leq c$ then we find $\phi$ solving
 \[\phi_{\overline{z}}-\mu\phi_z=\mu\, .\]
 Set $\Phi(z) = z+\phi(z)$. Then obviously 
 \[\Phi_{\overline{z}}=\mu\Phi_z\, \]
 and using Lemma \ref{linhombel} we find 
  \[ \|\Phi(z)-z\|_{1+\alpha} = \|\phi\|_{1+\alpha} \leq \bar{C}\|\mu\|_{\alpha}\, ,\]
 and 
 \[ \|D^k(\Phi(z)-z)\|_{1+\beta} = \|D^k\phi\|_{1+\beta} \leq C\|D^{k}\mu\|_{\beta}\]
for any $1\leq k\leq N$, which is what we wanted.
\end{proof}

\subsection{Proof of Proposition \ref{p:conformal}} Given the estimates of the previous paragraphs, Proposition \ref{p:conformal} can be proved following the classical approach, see for instance \cite[Addendum 1 to Chapter 9]{Spivak4}. We report however the argument for the reader's convenience. 

With a simple scaling argument we can assume $r=1$. Let $x,y$ be global coordinates on $\bar D_1$. Then $g$ takes the form
  \[g =  \xi dx^2+2 \zeta dxdy+\omega dy^2\, ,\]
for some functions $\xi,\zeta,\omega \in C^{N,\alpha}(D_1)$. We want to find a function $\Phi:D_1\to \R^2, (x,y)\mapsto \left(\Phi_1(x,y),\Phi_2(x,y)\right) \eqqcolon (s,t)$ such that in these new coordinates we have
  \[g = \rho^2\circ\Phi^{-1}(s,t)\left( ds^2+dt^2 \right)\, ,\]
hence 
  \[g=\rho^2\left( \left(\Phi_{1x}^2+\Phi_{2x}^2\right)dx^2+2\left(\Phi_{1x}\Phi_{1y}+\Phi_{2x}\Phi_{2y}\right)dxdy+ \left(\Phi_{1y}^2+\Phi_{2y}^2\right)dy^2 \right)\, ,\]
or 
  \begin{equation}\label{geq}
    g = \rho^2\left(\nabla\Phi_1\otimes\nabla\Phi_1+\nabla\Phi_2\otimes\nabla\Phi_2\right)\, .
  \end{equation}  
A comparison yields 
  \[\xi\omega-\zeta^2 = \rho^4\left( \Phi_{1x}\Phi_{2y}-\Phi_{1y}\Phi_{2x} \right)^2 = \rho^4J\Phi^2 \, ,\]
with $J\Phi = \det\nabla\Phi$. Consequently 
  \begin{equation}\label{rho}
    \rho^2 = \frac{\sqrt{\Delta}}{J\Phi}\, ,
  \end{equation}
where $\Delta = \xi\zeta-\omega^2$.\\
It is convenient to switch to complex notation. Consider $z=x+iy$, $\Phi(z) =\Phi_1(x,y)+i\Phi_2(x,y)$. A computation shows that \eqref{geq} is equivalent to the Beltrami equation for $\Phi$:
  \begin{equation}\label{mubeltrami}
    \Phi_{\overline{z}}(z) =\mu(z)\Phi_z(z), \quad z\in D_1\, ,
  \end{equation}
with the coefficient 
  \begin{equation}\label{mu0}
    \mu = \frac{\xi-\omega+2i\zeta}{\xi+\omega+2\sqrt{\Delta}}\, .
  \end{equation}
Now we extend $g$ to a symmetric $2\times 2$ tensor to $\R^2$ so that
  \begin{align*}
    \|g-e\|_{\alpha;\R^2} &\leq \bar{C}(\alpha)\|g-e\|_{\alpha;\bar D_1} \, ,\\ 
    \|g-e\|_{k+\beta;\R^{2}} & \leq C(N,\alpha,\beta)\|g-e\|_{k+\beta;\bar D_1}\, ,
  \end{align*}
for $1\leq k\leq N$. In particular note that if $\sigma_1$ is chosen sufficiently small, then $g\geq \frac{1}{2} e$ on the whole $\R^2$. Repeated applications of \eqref{e:chain3} and \eqref{e:product} to the 
expression \eqref{mu0} then yield
  \begin{align}
   \|\mu\|_{\alpha;\R^2} &\leq C\|g-e\|_{\alpha;D_1} \label{e:muestimatealpha} \, ,\\
   \|\mu\|_{k+\beta;\R^2}&\leq C\|g-e\|_{k+\beta;D_1} \label{e:muestimatebeta} \, ,
  \end{align}
where the former constant in \eqref{e:muestimatealpha} is a universal one and the latter depends only on $\alpha$, $\beta$ and $N$. Hence $\mu \in C^{N,\alpha}\left( \R^{2} \right)$.
Next we choose a $C^\infty$ cutoff function $\eta$ such that 
  \[ \eta(z) = 
     \begin{cases}
       1, & \text{ if } z\in \bar D_{1}\\
       0, & \text{ if } z\in \C\setminus D_{\frac{3}{2}}
     \end{cases}\]
With this define a new function
  \[\tilde{\mu} = \eta\mu\, .\]
By definition we have $\tilde{\mu} \in C^{N,\alpha}_c(D_2)$, thus by Corollary \ref{Beltrami} there exist constants $C$, $c$ and $\overline{C}$
such that if $\|\tilde{\mu}\|_{\alpha;D_2} \leq c$ then there exists $\Phi \in C^{N+1,\alpha}(\bar D_2)$ with
   \[\Phi_{\overline{z}}(z) =\tilde{\mu}(z)\Phi_z(z), \quad z\in D_2\, , \]
and
  \begin{align}
    \|\Phi(z)-z\|_{1+\alpha;D_2} &\leq \overline{C}\|\tilde{\mu}\|_{\alpha;D_2}\label{Phiestimate}\, ,\\
    \|D^k(\Phi(z)-z)\|_{1+\beta;D_2} &\leq C\|\tilde{\mu}\|_{k+\beta;D_2}\nonumber\, ,
  \end{align}
for any $1\leq k\leq N$. Observe that in particular $\Phi$ solves \eqref{mubeltrami}. Moreover, 
    \[\|\tilde{\mu}\|_{\alpha;D_2} \leq \|\mu\|_{\alpha;D_2}\|\eta\|_{\alpha;D_2}\leq C\|\mu\|_{\alpha;D_2} \leq C\|g-e\|_{\alpha;D_1}\, ,\]
and similarly
    \[\|\tilde{\mu}\|_{k+\beta;D_2} \leq C \|\mu\|_{k+\beta;D_2}\leq C\|g-e\|_{k+\beta;D_1} \, ,\]
by \eqref{e:muestimatealpha} and \eqref{e:muestimatebeta}. This shows that if $\|g-e\|_{\alpha;D_1}\leq \sigma_1$ with $\sigma_1$ small enough, we recover a coordinate change $\Phi$ solving \eqref{mubeltrami}. The estimates for $\Phi$ follow immediately. 
For the estimates of $\rho$ we use the fact that due to \eqref{Phiestimate} we have
  \[ (1-C\|g-e\|_{\alpha, \bar D_1})^2 \leq J\Phi \leq (1+C\|g-e\|_{\alpha, \bar D_1})^2\, ,\]
which together with the expression \eqref{rho}, the bounds on $\Phi$, \eqref{e:chain3} and \eqref{e:product} imply
  \[ \|D^k\rho\|_{\beta} \leq C\|g-e\|_{k+\beta; \bar D_1}\]
for $1\leq k\leq N$. This proves the claim. 

\bibliographystyle{acm}
\bibliography{iso}

\end{document}